\documentclass[abstract=true]{scrartcl}

\usepackage{amsmath}
\usepackage{amsthm}
\usepackage{amssymb}
\usepackage{amsfonts}
\usepackage{hyperref}
\usepackage{color}
\usepackage{dsfont}
\usepackage{enumitem}
\usepackage{algorithm}
\usepackage{algorithmic}
\usepackage{appendix}
\usepackage[font=small, labelfont=bf, labelfont+=sf]{caption}

\usepackage{yhmath} 

\usepackage{tikz-cd}
\tikzcdset{every label/.append style = {font = \large}}

\usepackage[numbers]{natbib}

\newtheoremstyle{curs} 
  {\topsep}            
  {\topsep}            
  {\itshape}           
  {0pt}                
  {\bfseries\sffamily} 
  {.}                  
  { }                  
  {}                   
  
\newtheoremstyle{ncurs} 
  {\topsep}             
  {\topsep}             
  {}                    
  {0pt}                 
  {\bfseries\sffamily}  
  {.}                   
  { }                   
  {}                    

\theoremstyle{ncurs}
\newtheorem{dfn}{Definition}[section]
\newtheorem{rmk}[dfn]{Remark}

\newtheorem{ass}[dfn]{Assumption}

\newtheorem{manualassumptioninner}{Assumption}
\newenvironment{manualassumption}[1]{%
  \renewcommand\themanualassumptioninner{#1}%
  \manualassumptioninner
}{\endmanualassumptioninner}

\theoremstyle{curs}
\newtheorem{lem}[dfn]{Lemma}

\newtheorem{cor}[dfn]{Corollary}
\newtheorem{thm}[dfn]{Theorem}

\numberwithin{equation}{section}

\def\d{\,\mathrm{d}} 
\def\P{\mathbf{P}}   
\def\D{\mathbf{D}}   
\def\Rk{\mathcal{R}} 
\def\O{\mathcal{O}}  
\def\cN{\mathcal{N}} 
\def\e{\mathrm{e}}   
\def\Lls{L_{\mathrm{LS}}} 
\def\Lclass{L_{\mathrm{class}}}

\def\Lhinge{L_{\mathrm{hinge}}}

\newcommand{\R}{\mathbb{R}}
\newcommand{\N}{\mathbb{N}}

\newcommand{\E}{\mathbf{E}}
\newcommand{\norm}[1]{\Vert#1\Vert}

\newcommand{\id}{\mathrm{id}}

\newcommand{\supp}{\mathrm{supp}}
\newcommand{\dist}{\mathrm{dist}}
\newcommand{\diam}{\mathrm{diam}}
\newcommand{\bgamma}{\boldsymbol{\gamma}}
\newcommand{\blambda}{\boldsymbol{\lambda}}
\newcommand{\LRKHS}{H_{\bgamma,\blambda}(\mathcal{A})}
\newcommand{\fftsvm}{\wideparen{f}_{D,\blambda,\bgamma,\mathrm{FFT}(m)}}
\newcommand{\fftsvmcv}{\wideparen{f}_{D_1,\blambda_{D_2},\bgamma_{D_2},\mathrm{FFT}(m)}}

\DeclareMathOperator*{\argmin}{arg\,min}

\makeatletter
\def\namedlabel#1#2{\begingroup
    #2%
    \def\@currentlabel{#2}%
    \phantomsection\label{#1}\endgroup
}
\makeatother

\title{Intrinsic Dimension Adaptive Partitioning for Kernel Methods}
\author{Thomas Hamm and Ingo Steinwart}
\date {\today}

\begin{document}
\maketitle
\begin{center}
Institute for Stochastics and Applications\\
Faculty 8: Mathematics and Physics\\
University of Stuttgart\\
D-70569 Stuttgart Germany\\
\texttt{\{\href{mailto:thomas.hamm@mathematik.uni-stuttgart.de}{thomas.hamm}, \href{mailto:ingo.steinwart@mathematik.uni-stuttgart.de}{ingo.steinwart}\}@mathematik.uni-stuttgart.de}
\end{center}
\begin{abstract}
We prove minimax optimal learning rates for kernel ridge regression, resp.~support vector machines based on a data dependent partition of the input space, where the dependence of the dimension of the input space is replaced by the fractal dimension of the support of the data generating distribution. We further show that these optimal rates can be achieved by a training validation procedure without any prior knowledge on this intrinsic dimension of the data. Finally, we conduct extensive experiments which demonstrate that our considered learning methods are actually able to generalize from a dataset that is non-trivially embedded in a much higher dimensional space just as well as from the original dataset.
\par\vskip\baselineskip\noindent
\end{abstract}
\section{Introduction}
In the theoretical analysis of learning methods it is a well-known fact that rates of convergence are significantly deteriorated by the dimension of the input space, a phenomenon usually denoted as the curse of dimensionality in statistical learning theory, exemplified by the well-known results on minimax optimal learning rates for regression by Stone \cite{Stone_GlobalOptimalRatesOfConvergenceForNonparametricRegression} and binary classification by Audibert and Tsybakov \cite{AudibertTsybakov_FastLearningRatesForPlugInClassifiers}. In addition, one state of the art learning methods, namely kernel methods, suffer from high computational costs, which are quadratic in space and at least quadratic in time. In this work we present a data dependent partitioning scheme for kernel methods that alleviates both limitations. The presented partitioning scheme is adaptive to the intrinsic dimension of the data and additionally reduces the space and time complexity. Since these two issues are significant on their own, they are usually treated independently in the literature, for which reason we briefly discuss these two topics separately.


As a consequence of the curse of dimensionality any learning method in a very high dimensional task is bound to fail  given any reasonably realistic sample size, at least in theory. In practice however, high dimensional datasets often exhibit a small \emph{intrinsic} dimensional structure as opposed to being uniformly spread out on the whole input space. Therefore, it is an interesting question whether common learning methods are able to exploit this, yet to be formalized, low intrinsic dimensional structure. There already exists a large amount of literature dealing with this problem ranging from tree based learning methods \cite{DasguptaFreund_RandomProjectionTreesAndLowDimensionalManifolds, KpotufeDasgupta_ATreeBasedRegressorThatAdaptsToIntrinsicDimension, ScottNowak_MinimaxOptimalClassificationWithDyadicDecisionTrees}, $k$-nearest neighbor \cite{Kpotufe_kNNRegressionAdaptsToLocalIntrinsicDimension, KulkarniPosner_RatesOfConvergenceOfNearestNeighborEstimationUnderArbitrarySampling}, Nadaraya-Watson kernel regression and local polynomial regression \cite{KpotufeGarg_AdaptivityToLocalSmoothnesAndDimensionInKernelRegression, BickelLi_LocalPolynomialRegressuinOnUnknownManifolds}, Bayesian regression using Gaussian processes \cite{YangDunson_BayesianManifoldRegression}, to support vector machines and kernel ridge regression \cite{HammSteinwart_IntrinsicDimension, McRaeRomberDavenport_SampleComplexitsAndEffectiveDimension, YeZhou_LearningAndApproximationByGaussiansOnRiemannianManifolds, YeZhou_SVMLearningAndLpApproximationByGaussianOnRiemannianManifolds}. The by far widest spread assumption to formalize the notion of intrinsic dimensionality of data is to assume that the data generating distribution is supported on a low-dimensional, smooth manifold. However, there is a considerable gap between the commonly accepted hypothesis that the data is not uniformly spread over the input space and the assumption, derived from this hypothesis, that the data lies on a smooth manifold. We shorten this gap by considering the \emph{fractal} dimension of the support of the data generating distribution, considerably weakening the manifold assumption similar to the recent results of \cite{HammSteinwart_IntrinsicDimension}. More precisely, we prove minimax optimal learning rates for regression and classification, where the dimension of the ambient space is replaced by fractal dimension of the support of the data generating distribution for kernel ridge regression, resp.~support vector machines using Gaussian kernels based on a data dependent partition of the input space. Although in this work we will only consider least-squares regression and binary classification using the hinge loss, our technique is flexible enough to handle general loss functions, see Theorem \ref{thm:general_oracle_inequality}.

Partitioning schemes are a popular approach for alleviating computational constraints for kernel methods. For example, in random chunking the dataset is randomly partitioned into $m$ subsets that are used for computing $m$ decision functions that have to be averaged for inference, see for example \cite{ZhangDuchiWainwright_DivideAndConquerKernelRidgeRegression}. In our approach we divide the input space into $m$ disjoint, geometrically defined cells and solve the initial objective on each cell independently using only the data points contained in that respective cell. Prediction for a new input is then performed using only the decision function of the cell in which the new input is contained. As a consequence, unlike the random chunking approach, our method enjoys also a speed-up during the testing stage since since our final estimator does not have to be averaged. This geometric partitioning scheme has already been studied, for example, in \cite{BlaschzykSteinwart_ImprovedClassificationRatesForLocalizedSVMs, MeisterSteinwartLSVMsrates, Muecke_ReducingTrainingTime}. In these works however, the authors consider an a-priori fixed partition of the input space satisfying some technical assumptions, whereas we consider a fully data dependent partition based on the farthest first traversal algorithm. Other popular approaches for speeding up kernel methods are Nystr\"om subsampling \cite{WilliamsSeeger_NystromMethod}, where a low rank approximation of the kernel matrix is used or random Fourier features \cite{RahimiRecht_RandomFeatures} where for translation invariant kernels a low dimensional, randomized approximation of the feature map is computed.


In the field of manifold learning much effort is put into designing algorithms that actively exploit some manifold structure of the data, such as computing a low dimensional representation of the data \cite{BelkinNiyogi_LaplacianEigenmaps} or manifold regularization \cite{BelkinNiyogi_ManifoldRegularization}. In contrast, our approach is completely agnostic to the assumption of a low intrinsic dimension of the data. Namely, we show that the optimal learning rates we derive can be achieved with a simple training validation procedure for hyperparameter selection. In summary, the resulting learning method, in terms of generalization performance, is adaptive to both, the regularity of the target function and the intrinsic dimension of the data, while at the same time providing a significant computational advantage. We further complement our theoretical findings with extensive experimental results, showing that regularized kernel methods using cross validation for hyperparameter selection can generalize just as well when the dataset is non-trivially embedded in a much higher dimensional space.

In the following we introduce in detail the framework for the rest of the paper. Assume we are given an input space $X\subset\R^d$ and an output space $Y\subset \R$ and an unknown probability distribution $\P$ on $X\times Y$. Our goal is to learn a functional relationship between $X$ and $Y$ based on a sample $D=((x_1,y_1),\ldots,(x_n,y_n))\in(X\times Y)^n$ drawn from $\P^n$, where our learning goal is specified by a loss function $L:Y\times \R\to[0,\infty)$. In this work we are concerned with regression using the least-squares loss $\Lls(y,t):=(y-t)^2$, where $Y\subset\R$ is an interval and binary classification using the hinge loss $\Lhinge(y,t):=\max\{0,1-yt\}$, where $Y=\{-1,1\}$. The quality of a decision function $f:X\to \R$ is measured by its risk
\begin{equation*}
\Rk_{L,\P}(f):=\int_{X\times Y} L(y,f(x))\,\mathrm{d}\P(x,y).
\end{equation*}
Furthermore, the minimum possible risk, denoted by $\Rk_{L,\P}^*:=\inf_{f:X\to\R} \Rk_{L,\P}(f)$, is called the Bayes risk and any function $f_{L,\P}^*$ with $\Rk_{L,\P}(f_{L,\P}^*)=\Rk_{L,\P}^*$ is called a Bayes decision function. As learning methods we will consider regularized empirical risk minimizers over reproducing kernel Hilbert spaces $H$ (see Appendix  \ref{sec:RKHS_fundamentals})
\begin{equation}
f_{D,\lambda}:=\argmin_{f\in H} \lambda\norm{f}_H^2+\frac{1}{n}\sum_{i=1}^nL(y_i,f(x_i))\label{eqn:def_SVM}
\end{equation}
with a regularization parameter $\lambda>0$. Note that by \cite[Theorem 5.5]{SteinwartChristmannSVMs} $f_{D,\lambda}$ exists and is unique for convex\footnote{We call a loss function $L:Y\times \R\to [0,\infty)$ convex if $L(y,\cdot)$ is convex for all $y\in Y$.} loss functions $L$. We briefly describe the spatially localized version of the estimator (\ref{eqn:def_SVM}), which will be introduced in detail in Section \ref{sec:loc_kernels}. Given a partition $\mathcal{A}=(A_j)_{j=1,\ldots,m}$ of the input space $X$, we compute $m$ independent decision functions on the cells $A_1,\ldots,A_m$ defined by the objective (\ref{eqn:def_SVM}) using only the data points contained in the respective cell. The prediction for a new input $x\in X$ is then performed by evaluating the decision function of the cell $A_j$ in which $x$ is contained.

The rest of this paper is organized as follows: Section \ref{sec:intrinsic_dim_ass} contains an introduction to our notion of intrinsic dimensionality. In Section \ref{sec:loc_kernels} we describe in detail the localization procedure and the construction of the partition. Sections \ref{sec:regression} and \ref{sec:classification} contain our main results for regression and classification, respectively. In Section \ref{sec:experimental_results} we present experimental results where we compare the generalization performance of the global and local version of our considered estimators using a dataset that is non-trivially embedded in a much higher dimensional space to the performance that is achieved when using the unmodified dataset.

\subsection*{Notation}
We denote the set of natural numbers $\{1,2,\ldots\}$ without zero by $\N$ and $\N_0=\{0\}\cup\N$. Given $x\in\R^d$ and $r>0$ then $B_r(x)=\{y\in\R^d:\norm{x-y}\leq r\}$ denotes the closed ball with radius $r$ and center $x$, where $\norm{\cdot}$ denotes the Euclidean norm. For a general normed space $E$, $B_E=\{x\in E:\norm{x}_E\leq 1\}$ denotes the closed unit ball centered at the origin. Given a measure space $(\Omega,\mathcal{F},\mu)$ we denote the usual Lebesgue-space of $p$-integrable functions by $L_p(\mu)$ for $p\in[1,\infty]$. For a set $X$, the space of bounded functions $f:X\to\R$ equipped with the sup-norm is denoted by $\ell_\infty(X)$.
\section{Intrinsic Dimensionality of Data}\label{sec:intrinsic_dim_ass}
The formalization of the notion of intrinsic dimensionality of the data, which is agreed upon in the literature, consists of describing it by the dimension of the support of the marginal $\P_X$ of the data generating distribution $\P$. As we describe the dimensionality of this support in terms of fractal dimensions, we first need to introduce the concept of covering and entropy numbers.
\begin{dfn}
Let $A\subset \R^d$. A subset $N\subset \R^d$ is called an $\varepsilon$-net of $A$ if $A\subset \bigcup_{x\in N}B_\varepsilon(x)$.
\begin{enumerate}
\item For $\varepsilon>0$ the covering number $\cN(A,\varepsilon)$ is defined as the minimum size of an $\varepsilon$-net of $A$.
\item For $m\in\N$ the entropy number $\varepsilon_m(A)$ is defined as
\begin{equation*}
\varepsilon_m(A):=\inf\{\varepsilon>0: \text{there exists an } \varepsilon\text{-net } N\subset A \text{ of } A \text{ with } |N|=m  \}.
\end{equation*}
\end{enumerate}
\end{dfn}
Note that, for technical reasons that becomes obvious later, only in the definition of $\varepsilon_m$ we require the net to be contained inside the considered set. As one would expect, there is a close connection between covering and entropy numbers, which we will discuss later. From now on let $\mu:=\P_X$ be the marginal distribution of $\P$ on $X$ and let $S:=\supp\,\mu$.
\begin{ass}\label{ass:assouad_dim}
The set $S$ is bounded and there exist constants $C_S\geq 1$ and $\varrho>0$ such that
\begin{equation}
\sup_{x \in S} \cN\big( B_r(x)\cap S,\varepsilon \big) \leq C_S\left( \frac{\varepsilon}{r} \right)^{-\varrho}\quad \text{for all } 0<\varepsilon\leq r.\label{eqn:assouad_dim}
\end{equation}
\end{ass}
The infimum  over all exponents $\varrho$ such that (\ref{eqn:assouad_dim}) in Assumption \ref{ass:assouad_dim} is satisfied, is known as the Assouad dimension of $S$, see \cite[Section 2.1]{Fraser_AssouadDimensionAndFractalGeometry}. The definition of Assouad dimension generalizes straightforward to general metric spaces and is used to characterize metric spaces that can be bi-Lipschitz embedded in a Euclidean space, see \cite{LehrbackTuominen_ANoteOnTheDimensionsOfAssouadAndAikawa}. The exponent $\varrho$ in Assumption \ref{ass:assouad_dim} is consistent with classical notions of dimensions, e.g.~the dimension of Euclidean spaces and smooth manifolds, which is a consequence of the basic properties of the Assouad dimension summarized in \cite[Section 2.4]{Fraser_AssouadDimensionAndFractalGeometry}. Note that by choosing $r>0$ sufficiently large, Assumption \ref{ass:assouad_dim} especially implies that $\cN(S,\varepsilon)\in \O(\varepsilon^{-\varrho})$ as $\varepsilon\to 0$ and by some basic properties of covering and entropy numbers the latter is equivalent to $\varepsilon_m(S)\in\O(m^{-1/\varrho})$ as $m\to\infty$. As we often have to switch between bounds on entropy numbers and covering numbers, we also formulate for convenience an extended version of Assumption \ref{ass:assouad_dim}, where we demand that the previously stated bound on the asymptotic of $\varepsilon_m(S)$ is satisfied for the same constant $C_S$ and that this bound on $\varepsilon_m(S)$ is sharp.
\begin{manualassumption}{\ref{ass:assouad_dim}*}\label{ass:assouad_dim_ast}
Let $S$ satisfy Assumption \ref{ass:assouad_dim} for the constants $C_S$ and $\varrho$ as well as
\begin{equation}
C_S^{-1}m^{-\frac{1}{\varrho}}\leq \varepsilon_m(S)\leq C_Sm^{-\frac{1}{\varrho}} \quad \text{for all } m\in\N. \label{eqn:ass_entropy_numbers}
\end{equation}
\end{manualassumption}
Further, we make an assumption on the data generating distribution $\mu$.
\begin{ass}\label{ass:marginal}
There exist constants $C_\mu\geq 1$ and $\delta>0$ such that
\begin{align*}
\inf_{x\in S} \mu(B_r(x))\geq C_\mu^{-1}\, r^{\delta} \quad \text{for all } 0<r\leq \diam\, S.
\end{align*}
\end{ass}
To give a quick example on typical values of the constant $\delta$ in Assumption \ref{ass:marginal}, assume that $X=[-1,1]^d$ and that $\P_X$ has a density with respect to the uniform distribution on $X$ bounded away from 0. Then Assumption \ref{ass:marginal} is satisfied for $\delta=d$. Note that for this example not only the density assumption on $\P_X$ is crucial, but also the geometry of the support of $\P_X$. For example, if $\P_X$ is the uniform distribution on a domain $X$ with cusps, then in general Assumption \ref{ass:marginal} is not fulfilled, at least not for $\delta=d$. Similar assumptions are common in level set estimation, see for example \cite{Cuevas_SetEstimation} for a survey or \cite[Remark 1]{AmbrosioColestaniVilla_OuterMinowskiContent} for an explicit construction of probability measures on sets $S\subset \R^d$, that are the image of a compact set $K\subset \R^{d'}, d'\leq d$ under a Lipschitz map satisfying Assumption \ref{ass:marginal} for $\delta=d'$. More generally, connections between properties of metric spaces described by their covering numbers (such as Assumption \ref{ass:assouad_dim}) and properties of measures on that space (especially how they act on balls) is a well-studied field in fractal geometry, see for example \cite[Chapter 1]{Heinonen_LecturesOnAnalysisOnMetricSpaces}. Particularly interesting for us is that, by \cite[Theorem 13.5]{Heinonen_LecturesOnAnalysisOnMetricSpaces}, if Assumption \ref{ass:assouad_dim} is satisfied for some $\varrho$, then for every $\delta>\varrho$ there exists a measure $\mu$ on $S$ satisfying Assumption \ref{ass:marginal} for this respective $\delta$. As an important consequence, the set of probability measures $\P$ satisfying both, Assumption \ref{ass:assouad_dim} and \ref{ass:marginal} is non-empty, even in the general case of non-integer $\varrho$.
\section{Localized Kernels and Construction of Partition}\label{sec:loc_kernels}
The approach of dividing the input space into disjoint cells and solving the initial learning problem independently on each cell with the data points contained in the respective cell is especially convenient for kernel methods from a mathematical perspective, since this procedure can be described by simply using a modified kernel, which we will explain in the following. We will only consider the Gaussian RKHS in the following, although the construction for general RKHSs is the same. In Appendix \ref{sec:RKHS_fundamentals} we collected some basic properties of RKHSs which are important for the rest of this section.

For $X\subset\R^d$ let $H_\gamma(X)$ denote the Gaussian RKHS on $X$ of width $\gamma>0$, that is, the RKHS associated to the kernel $k_\gamma(x,y)=\exp(-\gamma^{-2}\norm{x-y}^2)$ for $x,y\in X$. Given a partition $\mathcal{A}=(A_j)_{j=1,\ldots,m}$ of $X$ and $\blambda=(\lambda_1,\ldots,\lambda_m),\bgamma=(\gamma_1,\ldots,\gamma_m)\in(0,\infty)^m$ let $\LRKHS$ be the space of functions $f:X\to\R$ such that $f|_{A_j}\in H_{\gamma_j}(A_j)$ for all $j=1,\ldots,m$ equipped with the norm
\begin{equation*}
\norm{f}_{\LRKHS}^2:=\sum_{j=1}^m \lambda_j\norm{f|_{A_j}}_{H_{\gamma_j}(A_j)}^2.
\end{equation*}
Then $\LRKHS$ is a Hilbert space where the inner product is given by
\begin{equation*}
\langle f,g\rangle_{\LRKHS}=\sum_{j=1}^m \lambda_j \langle f|_{A_j},g|_{A_j}\rangle_{H_{\gamma_j}(A_j)},\quad f,g\in\LRKHS.
\end{equation*}
Moreover, $\LRKHS$ is an RKHS. To see this, we define $k:X\times X\to \R$ by
\begin{equation*}
k(x,y):=\sum_{j=1}^m \lambda_j^{-1}\mathbf{1}_{A_j}(x)k_{\gamma_j}(x,y)\mathbf{1}_{A_j}(y)
\end{equation*}
and verify the reproducing property
\begin{equation*}
\langle f,k(x,\cdot)\rangle_{\LRKHS}=\sum_{j=1}^m \mathbf{1}_{A_j}(x)\langle f|_{A_j},k_{\gamma_j}(x,\cdot)|_{A_j}\rangle_{H_{\gamma_j}(A_j)}=f(x).
\end{equation*}
For the RKHS $\LRKHS$ and a convex loss function $L:Y\times \R\to[0,\infty)$ we now consider the regularized empirical risk minimizer
\begin{equation}
f_{D,\blambda,\bgamma}:=\argmin_{f\in \LRKHS}  \norm{f}_{\LRKHS}^2+\frac{1}{n}\sum_{i=1}^n L(y_i,f(x_i))\label{eqn:def_LSVM},
\end{equation}
where $D=((x_1,y_1),\ldots,(x_n,y_n))\in(X\times Y)^n$ is a dataset. Note that compared to the global objective (\ref{eqn:def_SVM}), in (\ref{eqn:def_LSVM}) the regularization parameter $\lambda$ is now a component of the RKHS norm and can be chosen individually on each cell. If we define $I_j=\{i:x_i\in A_j\}$ for $j=1,\ldots,m$, we can rewrite the learning objective as
\begin{equation*}
f_{D,\blambda,\bgamma}=\argmin_{f\in\LRKHS} \sum_{j=1}^m \left( \lambda_j\norm{f|_{A_j}}_{H_{\gamma_j}(A_j)}^2+\frac{1}{n}\sum_{i\in I_j} L\left(y_i,f|_{A_j}(x_i)\right)\right)
\end{equation*}
and we see that the learning objective is minimized for $f\in\LRKHS$ if and only if 
\begin{equation*}
f|_{A_j}=\argmin_{g\in H_{\gamma_j}(A_j)}\frac{n\lambda_j}{n_j}\norm{g}_{H_{\gamma_j}(A_j)}^2+\frac{1}{n_j}\sum_{i\in I_j} L(y_i,g(x_i)),
\end{equation*}
where $n_j:=|I_j|$ and we assume that $n_j\geq 1$ for all $j=1,\ldots,m$, i.e.~every cell contains at least one data point. By the uniqueness of (\ref{eqn:def_SVM}) we can conclude that $f_{D,\blambda,\bgamma}|_{A_j}=f_{D_j,\lambda_j',\gamma_j}$, where $D_j= ((x_i,y_i))_{i\in I_j}$, $\lambda_j'=n\lambda_j/n_j$ and
\begin{equation*}
f_{D_j,\lambda_j',\gamma_j}=\argmin_{f\in H_{\gamma_j}(A_j)} \lambda_j'\norm{f}_{H_{\gamma_j}(A_j)}+\frac{1}{n_j}\sum_{i\in I_j}L(y_i,f(x_i))
\end{equation*}
is the regularized empirical risk minimizer using the standard Gaussian RKHS and the dataset $D_j$.

We will consider a Voronoi partition of the input space $X$, based on a set of center points $C=\{c_1,\ldots,c_m\}$, where the center points are a subset of the input vectors $V=\{x_1,\ldots,x_n\}$ of our dataset $((x_1,y_1),\ldots,(x_n,y_n))$. To this end, recall that in a Voronoi partition $\mathcal{A}=(A_j)_{j=1,\ldots,n}$ with respect to the centers $C=\{c_1,\ldots,c_m\}$ each cell $A_j$ consists of all the points $x\in X$ that have $c_j$ as the closest center, where we break ties in favor of a smaller index $j$ of the center $c_j$. Our center points $C$ are constructed by the farthest first traversal (FFT) algorithm, see Algorithm \ref{alg:FFT}. We will denote the learning method (\ref{eqn:def_LSVM}) using the partition into $m$ cells constructed in this manner by $\fftsvm$, indicating that the dataset which is used for constructing the partition is the same as the one used for computing the individual decision functions.

\begin{algorithm}
\caption{Farthest First Traversal}
\begin{algorithmic}
\REQUIRE $V=\{v_1,\ldots,v_n\}\subset \R^d, k\leq n$
\STATE $C\gets \{v_1\}$
\WHILE{$|C|<k$} \STATE{$C\gets C\cup \{c\}$ for $c\in V\backslash C$ with maximum distance to $C$}  \ENDWHILE
\RETURN $C$
\end{algorithmic}
\label{alg:FFT}
\end{algorithm}

A key property of the farthest first traversal algorithm, which will be crucial in our subsequent analysis, is that it produces an approximate solution of the metric $k$-center problem, see \cite[Theorem 4.3]{Har-Peled_GeometricApproximationAlgorithms}. To this end, recall that the objective in the metric $k$-center problem is to find a set $C\subset V$ with $|C|=k$, which minimizes
\begin{equation}\label{eqn:metric_kcenter_objective}
\max_{v\in V}\min_{c\in C}\norm{v-c},
\end{equation}
that is, to find a set of $k$ centers such that the maximum distance from any $v\in V$ to its closest center is minimized. Solving the metric $k$-center exactly is NP-hard, the solution computed by FFT is within a factor of 2 of the optimal value of (\ref{eqn:metric_kcenter_objective}) and can be computed in $\O(kn)$ time.


For our subsequent theoretical analysis we need to introduce a technicality for treating unbounded loss functions. To this end, we say that a loss function $L:Y\times \R\to[0,\infty)$ can be clipped at some value $M>0$, if $L(y,\wideparen{t})\leq L(y,t)$ for all $y\in Y$ and $t\in\R$, where
\begin{align*}
\wideparen{t}:=
\begin{cases}
-M &\text{ for } t<-M,\\
t &\text{ for } t\in[-M,M],\\
M &\text{ for } t>M.
\end{cases}
\end{align*}
is the clipped value of $t$ at $M$, see \cite[Definition 2.22]{SteinwartChristmannSVMs}. For loss functions $L$ that can be clipped at $M>0$ we apply the clipping operation point wise to our estimator (\ref{eqn:def_LSVM}) and denote the resulting decision function by $\wideparen{f}_{D,\blambda,\bgamma}$.

As usual, the optimal choices for the hyperparameters $\blambda,\bgamma$ depend on the unknown characteristics of the data generating distribution $\P$. To address this issue, we also consider a training validation procedure, that adaptively picks hyperparameters that have a small empirical error on a validation set, which was not used for training. To this end, we split our dataset $D$ into a training set $D_1:=((x_1,y_1),\ldots,(x_l,y_l))$ and a validation set $D_2:=((x_{l+1},y_{l+1}),\ldots,(x_n,y_n))$, where $l:=\lfloor n/2\rfloor+1$ and pick finite sets of candidate values $\Lambda_n,\Gamma_n$ for $\lambda_j$ and $\gamma_j$. We then compute the decision functions $f_{D_1,\blambda,\bgamma}$ for all $\blambda\in\Lambda_n^m,\bgamma\in \Gamma_n^m$ using the training set $D_1$ and pick the hyperparameters $\blambda_{D_2},\bgamma_{D_2}$ which perform best on the validation set $D_2$, that is our final decision function $\wideparen{f}_{D_1,\blambda_{D_2},\bgamma_{D_2}}$ is defined by
\begin{equation}\label{eqn:def_training_validation}
\sum_{i=l+1}^n L\left(y_i,\wideparen{f}_{D_1,\blambda_{D_2},\bgamma_{D_2}}(x_i)\right)=\min_{(\blambda,\bgamma)\in\Lambda_n^m\times\Gamma_n^m}\sum_{i=l+1}^n L\left(y_i, \wideparen{f}_{D_1,\blambda,\bgamma}(x_i)\right).
\end{equation}
Note that since using the estimator (\ref{eqn:def_LSVM}) is equivalent to using $m$ independent decision functions, the validation step is executed independently on each cell, which amounts to a total number of $m|\Lambda_n\times \Gamma_n|$ training runs that need to be performed, instead of $|\Lambda_n\times\Gamma_n|^m$. In Sections \ref{sec:regression} and \ref{sec:classification} we will show that it is sufficient for the candidate sets $\Lambda_n,\Gamma_n$ to grow logarithmically in the sample size $n$ in order to achieve optimal learning rates. In contrast, \cite{BlaschzykSteinwart_ImprovedClassificationRatesForLocalizedSVMs, MeisterSteinwartLSVMsrates} consider a similar training validation procedure for kernel partitioning methods, but they require the candidate sets to grow at least linearly in $n$, which makes the validation step computationally infeasible.
\begin{rmk}
In the results of Sections \ref{sec:regression} and \ref{sec:classification} the reader will notice, that the regularization parameters $\lambda_1,\ldots,\lambda_m$ and the bandwidths $\gamma_1,\ldots,\gamma_m$ are chosen identically on each cell, i.e.~$\lambda_1=\ldots=\lambda_m$ and $\gamma_1=\ldots=\gamma_m$. The reason for this is, that the asymptotically optimal choices for the parameters $\lambda_1,\ldots,\lambda_m$ and $\gamma_1,\ldots,\gamma_m$ are determined by the \emph{global} regularity properties of the data generating distribution $\P$. We illustrate this in the case of least squares regression, where regularity is measured by the smoothness of the Bayes decision function $f_{L,\P}^*$. Assume that $f_{L,\P}^*\in C^\alpha(X)$ (see Section \ref{sec:regression} for a definition), but restricted to some subset $X'\subset X$ the Bayes function $f_{L,\P}^*$ is $\alpha'$-times differentiable with $\alpha'\gg\alpha$. Then for cells $A_j$ contained in $X'$, this property of $f_{L,\P}^*$ can be utilized by an individual adjustment of the hyperparameters on these cells, and in turn improve the overall generalization performance of the estimator. However, the \emph{asymptotic} behavior of the excess risk (on which we focus in our theoretical results) can not be improved by this individual choice and is bottlenecked by the global degree of smoothness of $f_{L,\P}^*$. Additionally, an identical choice of $\lambda_1,\ldots,\lambda_m$ and $\gamma_1,\ldots,\gamma_m$ across all cells greatly simplifies the expressions in our theoretical analysis. For an example of a set of regularity assumptions in binary classification, where an individual choice of the hyperparameters on the cells actually leads to an improved asymptotic behavior of the learning rate, we refer to \cite{BlaschzykSteinwart_ImprovedClassificationRatesForLocalizedSVMs}.
\end{rmk}
\section{Least Squares Regression}\label{sec:regression}
Before we introduce our regularity assumptions, we quickly recall multi-index notation. For a multi-index $\nu=(\nu_1,\ldots,\nu_d)\in\N_0^d$ we set $|\nu|:=\nu_1+\ldots+\nu_d$ as well as $\nu!:=\nu_1!\cdot\ldots\cdot\nu_d!$ and given $x\in\R^d$ we set $x^\nu:=x_1^{\nu_1}\cdot\ldots\cdot x_d^{\nu_d}$. For a $k$-times continuously differentiable function $f:\R^d\to\R$ let
\begin{equation*}
\partial^\nu f(x)= \frac{\partial^{|\nu|}f}{\partial x_1^{\nu_1}\ldots \partial x_d^{\nu_d}}(x).
\end{equation*}
be the $k$-th order partial derivative.
\begin{dfn}
For $k\in\N_0$ and $\beta \in [0,1]$ let $C^{k,\beta}(\R^d)$ be the set of $k$-times continuously differentiable functions $f:\R^d\to\R$ with
\begin{equation*}
|f|_{C^{k,\beta}(\R^d)}:=\max_{|\nu|=k} \sup_{\substack{x,y\in\R^d\\x\neq y}} \frac{|\partial^\nu f(x)-\partial^\nu f(y)|}{\norm{x-y}^\beta}<\infty.
\end{equation*}
\end{dfn}
Regularity assumptions for least squares regression are usually formulated by differentiability properties of the Bayes function $f_{L,\P}^*(x)=\E(Y|X=x)$. Since in our setting $S=\supp\,\P_X$ is in general a set with empty interior, we first have to make some preliminary considerations in order to impose differentiability assumptions on $f_{L,\P}^*$. To this end, assume that for a function $f:S\to\R$ there exists a collection of functions $f_\nu:S\to\R$, $\nu \in \N_0^d, |\nu|\leq k$, where $f_0=f$ such that
\begin{equation}\label{eqn:whitney_condition_1}
f_\nu (x)=\sum_{|\nu+\iota|\leq k} \frac{f_{\nu+\iota}(y)}{\iota !}(x-y)^\iota + R_\nu (x,y)
\end{equation}
with $|f_\nu(x)|\leq C$ for all $x\in S$ and the residuals $R_\nu$ satisfy
\begin{equation}\label{eqn:whitney_condition_2}
R_\nu(x,y) \leq C\norm{x-y}^{k+\beta-|\nu|}
\end{equation}
for some $0<\beta \leq 1$ and all $x,y\in S$ and $|\nu|\leq k$ and some finite constant $C>0$. The obvious motivation for conditions (\ref{eqn:whitney_condition_1}) and (\ref{eqn:whitney_condition_2}) is that if $f\in C^{k,\beta}(\R^d)$ and $S$ has non-empty interior, then (\ref{eqn:whitney_condition_1}) and (\ref{eqn:whitney_condition_2}) are satisfied for the partial derivatives $f_\nu=\partial^\nu f$. By Whitney's extension theorem \cite[Chapter VI, Theorem 4]{Stein_SingularIntegralsAndDifferentiabilityPropertiesOfFunctions}, for a closed set $S$ with \emph{empty} interior any function $f:S\to\R$ satisfying (\ref{eqn:whitney_condition_1}) and (\ref{eqn:whitney_condition_2}) has an extension to a function $f_0\in C^{k,\beta}(\R^d)$. As a consequence, our regularity assumption on the Bayes decision function $f_{L,\P}^*$ being contained in $C^{k,\beta}(\R^d)$ only depends on the values of $f_{L,\P}^*$ on $S$. Moreover, $f_{L,\P}^*:\R^d\to \R$ can always be chosen compactly supported. In the subsequent results, besides $f_{L,\P}^*\in C^{k,\beta}(\R^d)$, for some technical reasons we will also require that $f_{L,\P}^*\in L_2(\R^d)\cap L_\infty(\R^d)$. By the argument above, this poses no further restriction.


The following theorem, which is our first main result of this section, states an oracle inequality for the estimator (\ref{eqn:def_LSVM}) under Assumptions \ref{ass:assouad_dim_ast} and \ref{ass:marginal}.
\begin{thm}\label{thm:ls_oracle_inequality}
Let $\P$ satisfy Assumption \ref{ass:assouad_dim_ast} for the constants $C_S$ and $\varrho$ and Assumption \ref{ass:marginal} for the constants $C_\mu$ and $\delta$. Further assume that $Y\subset [-M,M]$ as well as $f_{L,\P}^*\in C^{k,\beta}(\R^d)\cap L_2(\R^d)\cap L_\infty (\R^d)$ for some $k\in\N_0$ and $\beta\in[0,1)$ and set $\alpha:=k+\beta$. Consider the estimator $\fftsvm$ using the least-squares loss $L=\Lls$ for some $m\leq n $ and hyperparameters $\lambda_1=\ldots=\lambda_m=:\lambda$ and $\gamma_1=\ldots=\gamma_m=:\gamma$. Then for all $\tau>0, n> 1, \lambda\in(0,1) $, and $ \gamma\in(0,m^{-1/\varrho}]$ we have
\begin{align*}
\Rk_{L,\P}(\fftsvm)-\Rk_{L,\P}^*\leq& c_1\norm{f_{L,\P}^*}_{L_2(\R^d)}^2m\lambda\gamma^{-d}+c_2|f_{L,\P}^*|_{C^{k,\beta}(\R^d)}^2\gamma^{2\alpha} \\
&+c_{m,n} K\lambda^{-1/\log n}\gamma^{-\varrho}n^{-1}\log^{d+1}n+c_3\frac{\tau}{n}
\end{align*}
with probability $\P^n$ not less than $1-3\e^{-\tau}$, where $c_1=9\pi^{-d/2}4^{k+1}$,
\begin{align*}
c_2&=9\left( \frac{\Gamma\left(\frac{\alpha+d}{2} \right)}{\Gamma\left( \frac{d}{2}\right)} \right)^2 2^{-\alpha}d^{k}, \\
c_{m,n}&=\max\left\{ C_S^{\varrho+1}\left(1+m^{2}\exp\left( -\log 2C_\mu^{-1} C_S^{-\delta}nm^{-\delta/\varrho}/\log n\right)\right),4M^2\right\}\max\left\{16M^2,1 \right\}, \\
c_3&=3456M^2+15\max\left\{ (2^{k+1}\norm{f_{L,\P}^*}_{L_\infty(\R^d)}+M)^2,4M^2\right\},
\end{align*}
and $K$ is a constant independent of $\P,\lambda,\gamma,n$, and $m$.
\end{thm}
Using Theorem \ref{thm:ls_oracle_inequality} we can easily deduce learning rates for the estimator (\ref{eqn:def_LSVM}) by simply specifying appropriate values for the regularization parameter $\lambda$ and the bandwidth $\gamma$.
\begin{cor}\label{cor:ls_learning_rates}
Let the assumptions of Theorem \ref{thm:ls_oracle_inequality} be satisfied with the number of cells specified as $m=\lceil n^\sigma\rceil$ for some $\sigma <1$. Assume that $|f_{L,\P}^*|_{C^{k,\beta}(\R^d)}\leq C_1$, $\norm{f_{L,\P}^*}_{L_2(\R^d)}\leq C_2$, and $\norm{f_{L,\P}^*}_{L_\infty(\R^d)}\leq C_3$ for some finite constants $C_1,C_2,C_3$ and that the parameters from Theorem \ref{thm:ls_oracle_inequality} satisfy
\begin{equation*}
\sigma<\min\left\{\frac{\varrho}{2\alpha+\varrho},\frac{\varrho}{\delta}\right\}.
\end{equation*}
Setting $\gamma=n^{-a} $ and $ \lambda=n^{-b}$ with $a=1/(2\alpha+\varrho) $ and $ b\geq \sigma+(2\alpha+d)/(2\alpha+\varrho)$ there exists a constant $C>0$ only depending on $C_1,C_2,C_3,C_\mu,C_S$ and $M$ such that for all $n>1$ and $\tau\geq 1$ we have
\begin{align*}
\Rk_{L,\P}(\fftsvm)-\Rk_{L,\P}^*\leq C\tau\,n^{-\frac{2\alpha}{2\alpha+\varrho}}\log^{d+1}n,
\end{align*}
with probability $\P^n$ not less than $1-\e^{-\tau}$.
\end{cor}
Note that the learning rate in Corollary \ref{cor:ls_learning_rates} coincides up to the logarithmic factor with the well-known optimal rate \cite{Stone_GlobalOptimalRatesOfConvergenceForNonparametricRegression}, which covers the case $\varrho=d$. In \cite[Remark 3.6]{HammSteinwart_IntrinsicDimension} it is described how to leverage this result to the case where $\varrho<d$ is an integer and deduce optimality of the rate in Corollary \ref{cor:ls_learning_rates} in these cases. In the elementary setting where $d=\varrho=\delta$ the constraint in Corollary \ref{cor:ls_learning_rates} becomes $\sigma< d/(2\alpha+d)$ which is similar to the constraint in \cite[Theorem 5]{MeisterSteinwartLSVMsrates} on the radii of the cells.

As the choice of $\lambda$ and $\gamma$ in the corollary above requires knowledge on the unknown parameters $\varrho$ and $\alpha$ of the distribution $\P$, we also present the following theorem, which states that the training validation procedure (\ref{eqn:def_training_validation}) achieves the same rates of Corollary \ref{cor:ls_learning_rates} without any knowledge on $\P$.
\begin{thm}\label{thm:ls_adaptive_rates}
Let the assumptions of Theorem \ref{thm:ls_oracle_inequality} be satisfied with $\varrho\geq 1$ and the number of cells specified as $m=\lceil n^\sigma\rceil$ for some $\sigma < 1$. Let $A_n$ be a minimal $1/\log n$-net of $(0,1]$ with $1\in A_n$ and let $B_n$ be a minimal $1/\log n$-net of $[\sigma+1,\sigma+d]$ with $\sigma+d\in B_n$ and set $\Gamma_n :=\{n^{-a}:a\in A_n \}$ and $\Lambda_n:=\{ n^{-b}:b\in B_n \}$. Assume that $|f_{L,\P}^*|_{C^{k,\beta}(\R^d)}\leq C_1$, $\norm{f_{L,\P}^*}_{L_2(\R^d)}\leq C_2$, and $\norm{f_{L,\P}^*}_{L_\infty(\R^d)}\leq C_3$ for some finite constants $C_1,C_2,C_3$ and that the parameters from Theorem \ref{thm:ls_oracle_inequality} satisfy
\begin{equation*}
\sigma<\min\left\{\frac{\varrho}{2\alpha+\varrho},\frac{\varrho}{\delta}\right\}.
\end{equation*}
Then there exists a constant $C>0$ only depending on $C_1,C_2,C_3,C_\mu,C_S$ and $M$ such that for all $n>\exp \frac{2\varrho}{\varrho-\sigma}$ and $\tau\geq 1$ we have
\begin{equation*}
\Rk_{L,\P}(\fftsvmcv)-\Rk_{L,\P}^*\leq C\tau\,n^{-\frac{2\alpha}{2\alpha+\varrho}}\log^{d+1}n
\end{equation*}
with probability $\P^n$ not less than $1-\e^{-\tau}$ where the hyperparameters $ \blambda_{D_2},\bgamma_{D_2}$ are selected by the training validation procedure (\ref{eqn:def_training_validation}).
\end{thm}
\begin{rmk}\label{rmk:speedup_accuracy_tradeoff}
The constraint on $\sigma$ in Corollary \ref{cor:ls_learning_rates} and Theorem \ref{thm:ls_adaptive_rates} can be interpreted as follows: The user specifies a value for $\sigma$ depending on some computational time and space constraints. The condition on $\sigma$ then specifies the set of distributions for which we can achieve the optimal learning rate. The smaller we choose $\sigma$, the larger this class of distributions becomes but small values of $\sigma$ in turn diminish the computational speed up. As a consequence we have a fundamental trade-off between computational benefit and the fastest achievable learning rate. Indeed, the fastest possible learning rate is given by $n^{\sigma-1}$, which can be seen by the bounds on the statistical error of the validation step in the proof of Theorem \ref{thm:ls_adaptive_rates}. A very similar trade-off was observed for least-squares kernel regression using random features \cite{RudiRosasco_RandomFeatures}, Nystr\"om \cite{RudiCamorianoRosasco_NystromRegularization} subsampling, and random chunking \cite{ZhangDuchiWainwright_DivideAndConquerKernelRidgeRegression} as speed up strategies. In all these articles the authors consider a more abstract setting of general kernels with assumptions on the decay of the eigenvalues of the corresponding integral operator, however, they focus on the restrictive case where the Bayes decision function is assumed to be contained in the considered RKHS. None of these mentioned articles consider adaptive hyperparameter selection for achieving the same rates without knowledge on the data generating distribution.
\end{rmk}
%
%
%
\section{Binary Classification}\label{sec:classification}
Throughout this section let $Y=\{-1,1\}$ and fix a version $\eta:X\rightarrow[0,1]$ of the posterior probability of $\P$, that is the probability measures $\P(\,\cdot\,|x)$ on $Y$ defined by $\P(\{1\}|x)=\eta(x)$ for $x\in X$ fulfill
\begin{align*}
\P(A\times B)=\int_A \P(B|x) \d\mu(x)
\end{align*}
for all measurable sets $A\subset X$ and $B\subset Y$. With the posterior probability $\eta$ the optimal labeling strategy $f_{\Lclass,\P}^*$ can be expressed by $f_{\Lclass,\P}^*(x)=\mathrm{sign}(2\eta(x)-1)$, cf.~\cite[Example 2.4]{SteinwartChristmannSVMs}.

One common way in binary classification to formulate regularity assumptions on $\P$ is to impose smoothness assumptions on the conditional probability function $\eta$. However, this type of regularity assumption is mainly suited for plug-in classifiers, which first compute a regression estimator $\widehat{\eta}$ of $\eta$ and then use $\mathrm{sign}(2\widehat{\eta}(x)-1) $ as final estimator. In contrast, we will describe our regularity assumptions by so-called margin or noise conditions. To this end, first note that it is intuitively hard to make a correct prediction for $x\in X$ if $\eta(x)\approx 1/2$. The following assumption captures this intuition by restricting the mass of points $x$ such that $\eta(x)$ is close to $1/2$. 
\begin{ass}\label{as:TNC}
There exist constants $C_*>0$ and $q\in[0,\infty]$ such that
\begin{align*}
\P_X\big(\{x\in X: |2\eta(x)-1|<t\}\big)\leq (C_* t)^q \quad \text{for all } t\geq 0.
\end{align*}
\end{ass}
Assumption \ref{as:TNC} is in the literature usually known as Tsybakov noise condition, see \cite{MammenTsybakov_SmoothDiscriminationAnalysis, Tsybakov_OptimalAggregationOfClassifiersInStatisticalLearning}. Note that Assumption \ref{as:TNC} only restricts the total mass of points $x$ with $\eta(x)\approx 1/2$ and not their location. Our second regularity assumption will additionally incorporate the distance of a point $x\in X$ to the decision boundary. To this end, we need the following
\begin{dfn}
Let $X_{-1}:=\{x\in X:\eta(x)<1/2\}$ and $X_1:=\{x\in X:\eta(x)>1/2\}$ and define
\begin{align*}
\Delta(x):=\begin{cases}
\dist(x,X_1) &\text{ if } x\in X_{-1},\\
\dist(x,X_{-1}) &\text{ if } x\in X_1,\\
0&\text{ else,}
\end{cases}
\end{align*}
where $\dist(x,A):=\inf_{y\in A}\norm{x-y}$.
\end{dfn}
Our second regularity assumption then reads as follows:
\begin{ass}\label{as:MNE}
There exist constants $C_{**}>0$ and $\beta >0$ such that
\begin{align*}
\int_{\{x\in X:\Delta(x)<t\}}|2\eta(x)-1|\d\P_X(x)\leq C_{**}t^\beta \quad \text{for all } t\geq 0.
\end{align*}
\end{ass}
Assumption \ref{as:MNE} was introduced in \cite{SteinwartScovel_FastRatesForSupportVectorMachinesUsingGaussianKernels} in a slightly different form and later in the presented form in \cite[Definition 8.15]{SteinwartChristmannSVMs}. Assumption \ref{as:MNE} was also used in \cite{HammSteinwart_IntrinsicDimension} to derive learning rates for classification based on the intrinsic (box-counting) dimension of the support of the data generating distribution. For a short instructive introduction to Assumptions \ref{as:TNC} and \ref{as:MNE}, including some examples, we also refer to \cite[Section 4]{HammSteinwart_IntrinsicDimension}.

The following theorem states states an oracle inequality for the estimator (\ref{eqn:def_LSVM}) using the hinge loss, which is the first main result in this section.
\begin{thm}\label{thm:class_oracle_inequality}
Let $\P$ satisfy Assumption \ref{ass:assouad_dim_ast} for the constants $C_S$ and $\varrho$ and Assumption \ref{ass:marginal} for the constants $C_\mu$ and $\delta$. Further let Assumption \ref{as:TNC} be satisfied for the constants $C_*$ and $q$ and Assumption \ref{as:MNE} for the constants $C_{**}$ and $\beta$. Consider the estimator $\fftsvm$ using the hinge loss $L=\Lhinge$ for some $m\leq n$ and hyperparameters $\lambda_1=\ldots=\lambda_m=:\lambda$ and $\gamma_1=\ldots=\gamma_m=:\gamma$. Then for all $\tau>0, n> 1,\lambda\in(0,1] $, and $ \gamma\in[n^{-1/\varrho},m^{-1/\varrho}]$ we have
\begin{align*}
&\Rk_{L,\P}(\fftsvm)-\Rk_{L,\P}^* \\
\leq& c_1m\lambda\gamma^{-d}+c_2C_{**}\gamma^\beta+c_{m,n}K\lambda^{-1/\log n}\left( \frac{\gamma^{-\varrho}}{n}\right)^\frac{q+1}{q+2}\log^{d+1}n \\
&+3C_*^\frac{q}{q+2}\left(\frac{432\tau}{n} \right)^\frac{q+1}{q+2}+30\frac{\tau}{n}
\end{align*}
with probability not less than $1-3\e^{-\tau}$, where $c_1=3^{d+2}/\Gamma(d/2+1)$,
\begin{align*}
c_2&=9\frac{2^{1-\beta/2}\Gamma\left( \frac{\beta+d}{2}\right)}{\Gamma(d/2)}, \\
c_{m,n}&=C_S^{\varrho+1}\left(1+m^2\exp\left( -\log 2 C_\mu^{-1} C_S^{-\varrho}nm^{-\delta/\varrho}/\log n\right)\right)\max\left\{ C_*^{q/(q+1)},1\right\},
\end{align*}
and $K$ is a constant independent of $\P,\lambda,\gamma,n$, and $m$.
\end{thm}
As in the last section, we can now derive learning rates for the estimator (\ref{eqn:def_LSVM}) simply by specifying values for the regularization parameter $\lambda$ and the bandwidth $\gamma$.
\begin{cor}\label{cor:class_learning_rates}
Let the assumptions of Theorem \ref{thm:class_oracle_inequality} be satisfied with the number of cells specified as $m=\lceil n^\sigma\rceil$ for some $\sigma< 1$. Assume the parameters from Theorem \ref{thm:class_oracle_inequality} satisfy
\begin{equation*}
\sigma < \min\left\{ \frac{\varrho(q+1)}{\beta(q+2)+\varrho(q+1)},\frac{\varrho}{\delta}\right\}.
\end{equation*}
Setting $\gamma_n=n^{-a}$ and $\lambda_n=n^{-b}$ with
\begin{equation*}
a= \frac{q+1}{\beta(q+2)+\varrho(q+1)}  \quad \text{ and } \quad  b\geq\sigma + \frac{(d+\beta)(q+1)}{\beta(q+2)+\varrho(q+1)}
\end{equation*}
there exists a constant $C>0$ only depending on $C_*, C_{**}, C_\mu$, and $C_S$ such that for all $n>1$ and $\tau \geq 1$ we have
\begin{align*}
\Rk_{L,\P}(\fftsvm)-\Rk_{L,\P}^*\leq C\tau\, n^{-\frac{\beta(q+1)}{\beta(q+2)+\varrho(q+1)}}\log^{d+1}n
\end{align*}
with probability $\P^n$ not less than $1-\e^{-\tau}$.
\end{cor}
For a quick comparison with existing results, we remark that if $\eta$ is $\alpha$-H\"older continuous for some $\alpha\in(0,1]$ and Assumption \ref{as:TNC} is fulfilled for some exponent $q>0$, then Assumption \ref{as:MNE} is fulfilled for the exponent $\beta=\alpha(q+1)$, see \cite[Proposition 4.4]{HammSteinwart_IntrinsicDimension} and the remark thereafter. In this case, the exponent in the learning rate of Corollary \ref{cor:class_learning_rates} is given by $\alpha(q+1)/(\alpha(q+2)+\varrho)$, which is the minimax optimal rate of convergence by \cite[Theorem 4.1]{AudibertTsybakov_FastLearningRatesForPlugInClassifiers}. Also note that in the special case $\varrho=d $ we sligthly improve the results of \cite{ThomannBlaschzykMeisterSteinwart_SpatialDecomposition} and \cite[Theorem 8.26]{SteinwartChristmannSVMs}, where the authors achieve the same rate as in Corollary \ref{cor:class_learning_rates} but with a sub optimality term of the form $n^\epsilon$ with $\epsilon>0$ arbitrarily small instead of our factor of $\log^{d+1}n$. In \cite{ThomannBlaschzykMeisterSteinwart_SpatialDecomposition} the authors analyze the same estimator as the one considered in Corollary \ref{cor:class_learning_rates}, albeit they consider an a priori fixed partition of the input space satisfying some technical assumptions, while we consider a fully data dependent partition. In \cite[Theorem 8.26]{SteinwartChristmannSVMs} an unmodified Gaussian SVM is considered.

The content of the following theorem is that this rate can again be achieved adaptively.
\begin{thm}\label{thm:class_adaptive_rates}
Let the assumptions of Theorem \ref{thm:class_oracle_inequality} be satisfied for $\varrho\geq 1$ and the number of cells specified as $m=\lceil n^\sigma\rceil$ for some $\sigma < 1$. Let $A_n$ be a minimal $1/\log n$-net of $(0,1]$ and let $B_n$ be a minimal $1/\log n$-net of $[\sigma+1,\sigma+d]$ with $\sigma+d\in B_n$ and set $\Gamma_n :=\{n^{-a}:a\in A_n \}$ and $\Lambda_n:=\{ n^{-b}:b\in B_n \}$. Assume that the parameters from Theorem \ref{thm:class_oracle_inequality} satisfy
\begin{equation*}
\sigma < \min\left\{ \frac{\varrho(q+1)}{\beta(q+2)+\varrho(q+1)},\frac{\varrho}{\delta}\right\}.
\end{equation*}
Then there exists a constant $C>0$ only depending on $C_{*,**},C_\mu$ and $C_S$ such that for all $n>\exp \frac{2\varrho+\log 2}{1-\sigma}$ and $\tau\geq 1$ we have
\begin{equation*}
\Rk_{L,\P}(\fftsvmcv)-\Rk_{L,\P}^*\leq C\tau\, n^{-\frac{\beta(q+1)}{\beta(q+2)+\varrho(q+1)}}\log^{d+1}n
\end{equation*}
with probability $\P^n$ not less than $1-\e^{-\tau}$ where the hyperparameters $ \blambda_{D_2},\bgamma_{D_2}$ are selected by the training validation procedure (\ref{eqn:def_training_validation}).
\end{thm}
Note that in Corollary \ref{cor:class_learning_rates} and Theorem \ref{thm:class_adaptive_rates} we are confronted with the same trade-off between computational speed up and generalization performance mentioned in Remark \ref{rmk:speedup_accuracy_tradeoff}. The same effects were observed in the results of \cite{BlaschzykSteinwart_ImprovedClassificationRatesForLocalizedSVMs,ThomannBlaschzykMeisterSteinwart_SpatialDecomposition}.

\section{Experimental Results}\label{sec:experimental_results}

In this section we complement our theoretical findings by experimentally verifying that, given some dataset $D$, learning method (\ref{eqn:def_LSVM}) achieves the same generalization performance if this dataset is non-trivially embedded in a much higher dimensional space. To this end, we define and embedding $\Phi:\R^d\to \R^{d+p}$ as follows: Sample $w_1,\ldots,w_p$ iid from the uniform distribution on $[-\pi,\pi]^d$ and define the function $\varphi:\R^d\to\R^p$ by $\varphi_j(x)=\sin \langle x,w_j\rangle$ for $j=1,\ldots,p$. Subsequently, sample an orthogonal matrix $T\in\R^{(d+p)\times(d+p)}$ from the Haar-measure and set $\Phi(x):=T(x,\varphi(x))$. Now, given a dataset $D=((x_1,y_1),\ldots,(x_n,y_n))$ with $x_i\in[-1,1]^d$ for $i=1,\ldots,n$ we define the embedded dataset $D_p=(\Phi(x_i),y_i)_{i=1,\ldots,n}$. The resulting dataset lies on the rotated graph of the map $\varphi:\R^d\to\R^p$ and therefore naturally has a non-trivial $d$-dimensional structure in a $(d+p)$-dimensional Euclidean space, see Figure \ref{fig:embedding}.


\begin{figure}[h!]
\includegraphics[width=.24\linewidth]{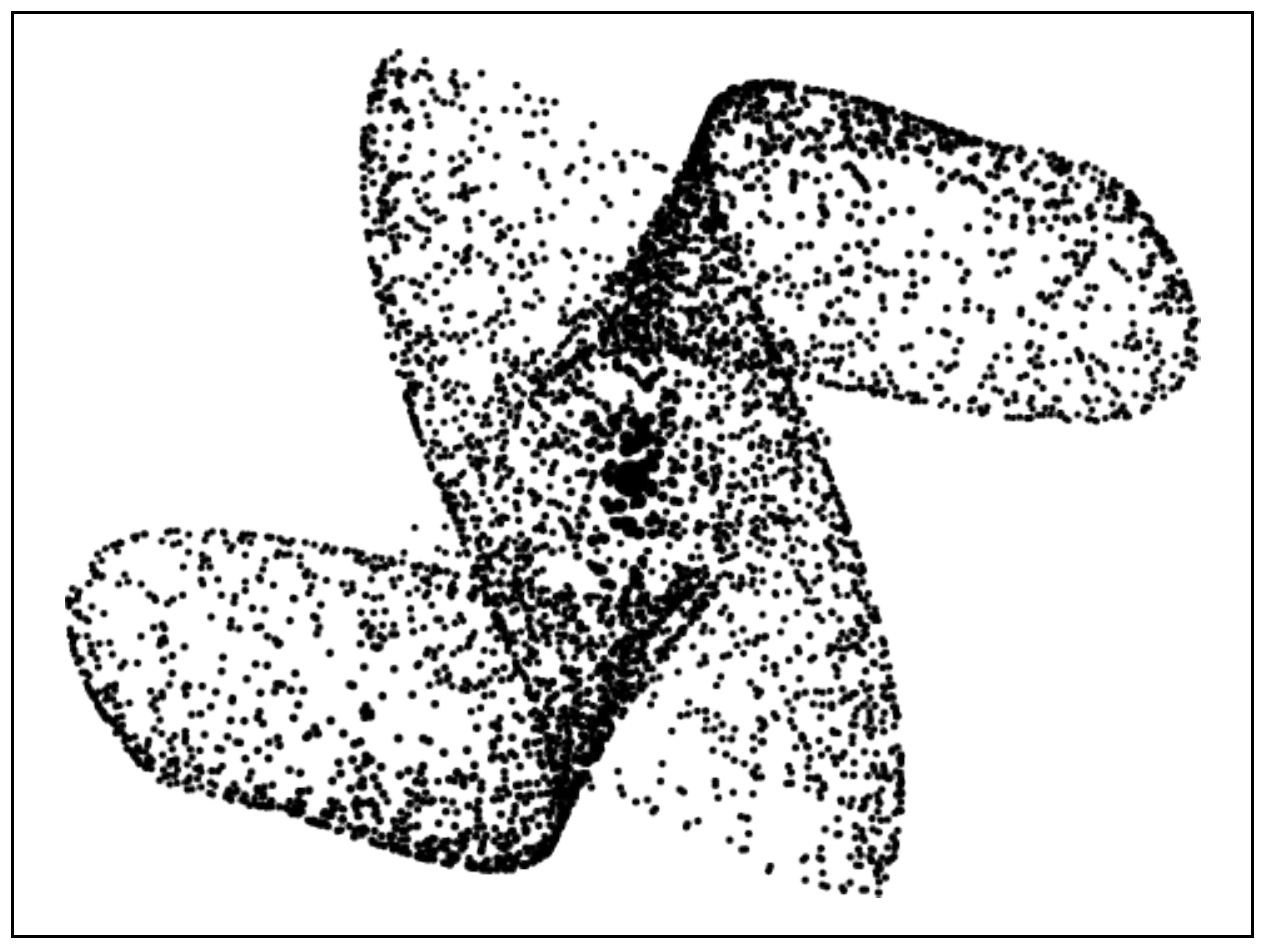}\hfill
\includegraphics[width=.24\linewidth]{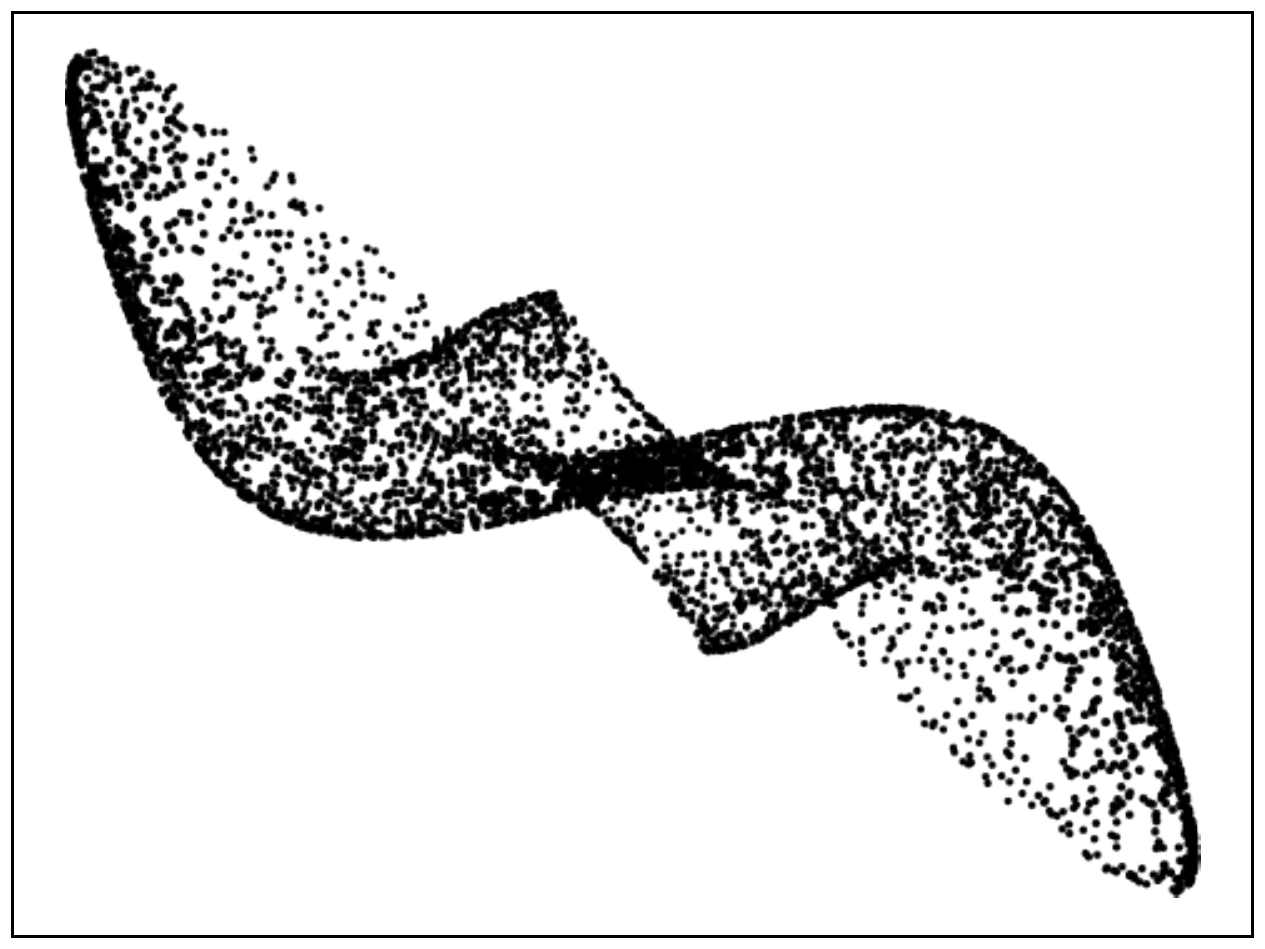}\hfill
\includegraphics[width=.24\linewidth]{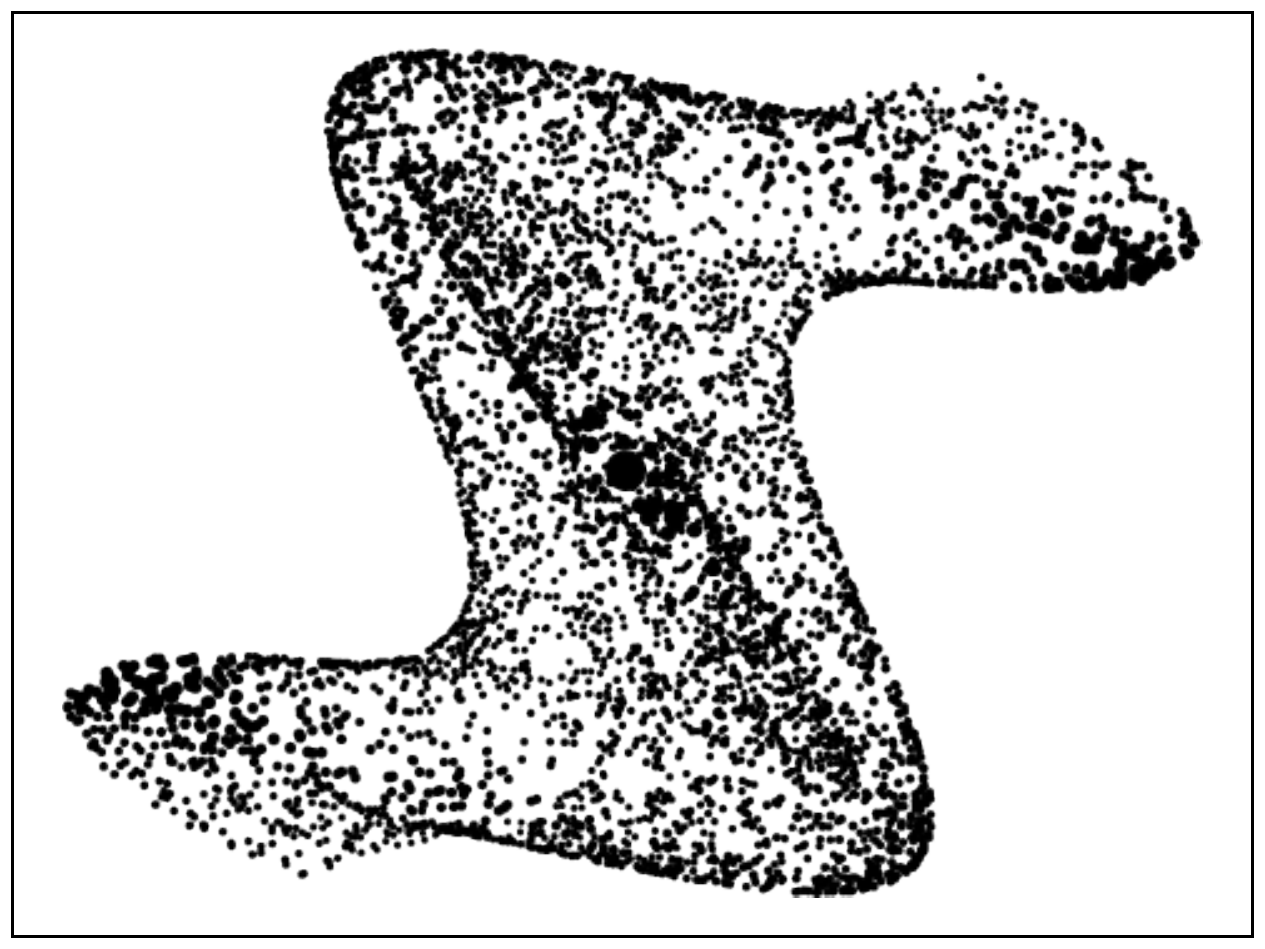}\hfill
\includegraphics[width=.24\linewidth]{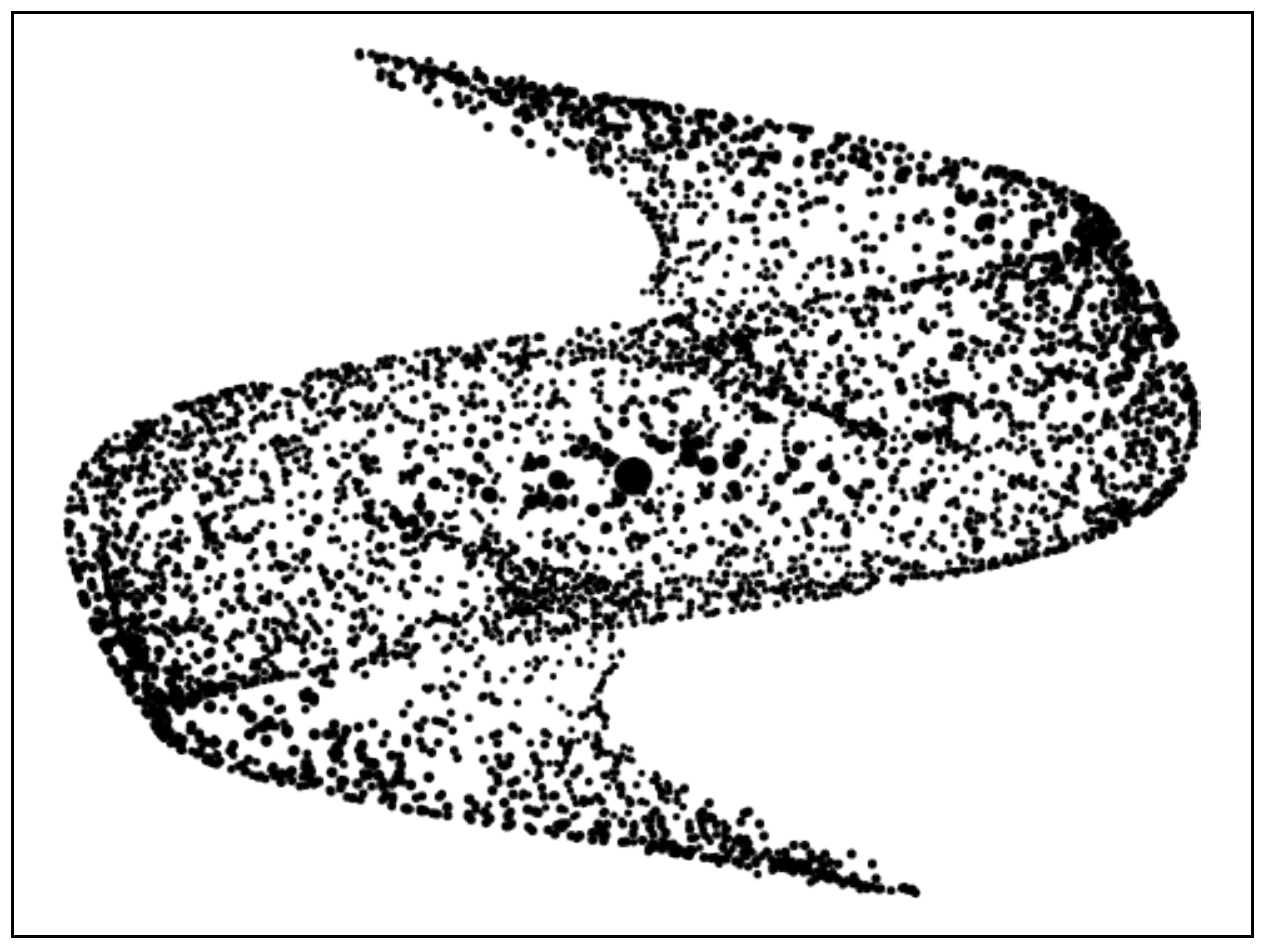}\hfill
\\[\smallskipamount]
\includegraphics[width=.24\linewidth]{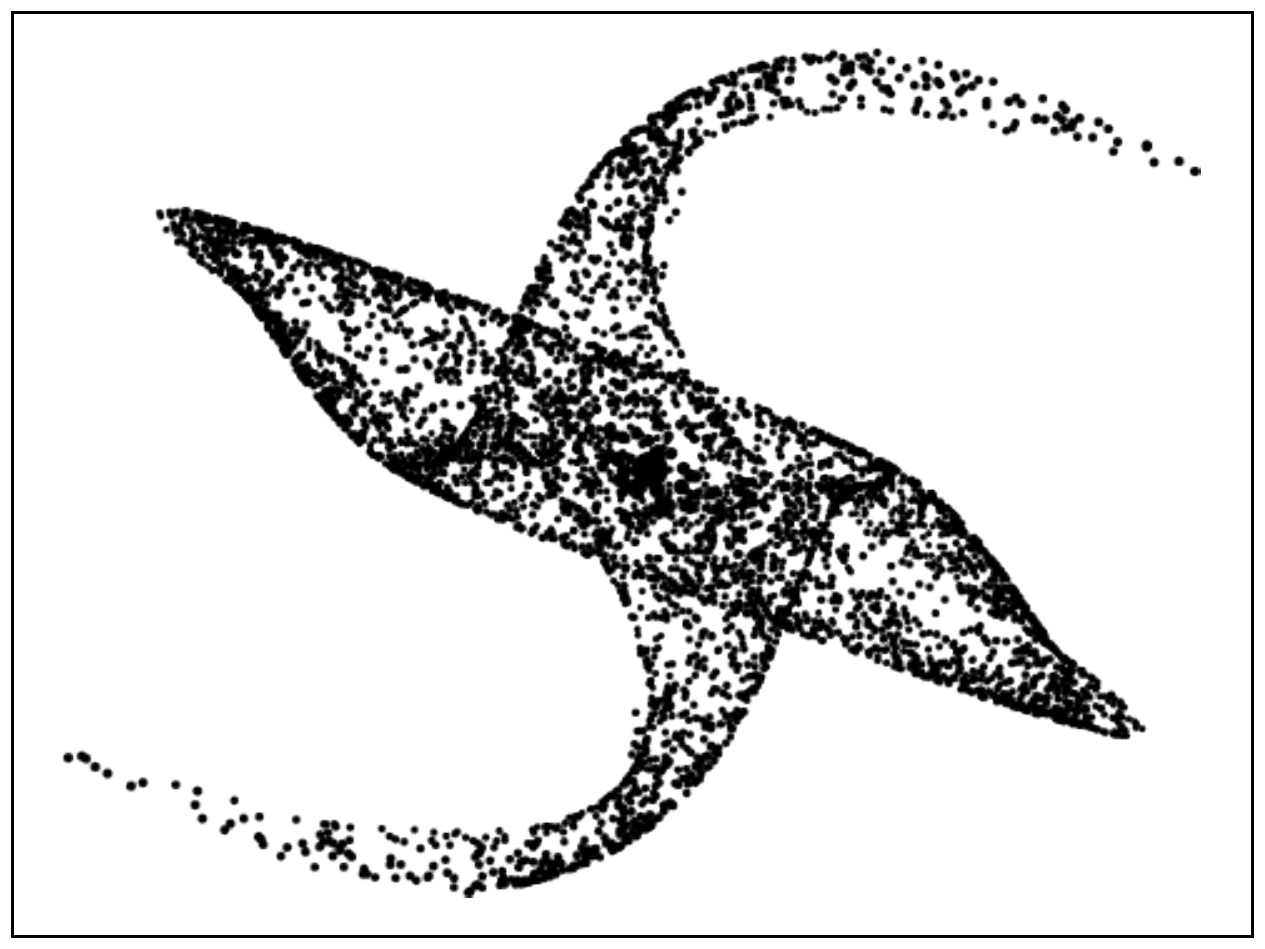}\hfill
\includegraphics[width=.24\linewidth]{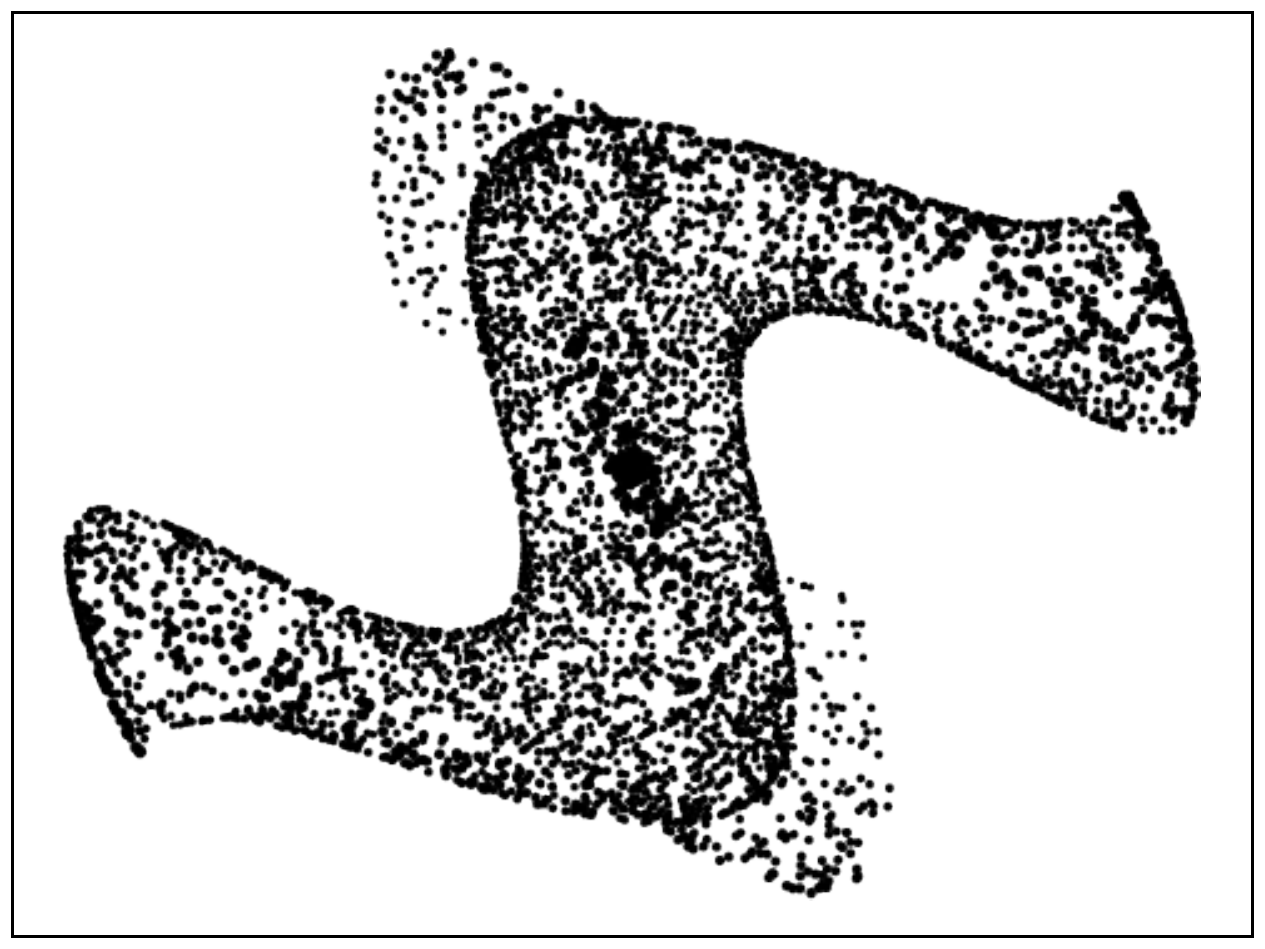}\hfill
\includegraphics[width=.24\linewidth]{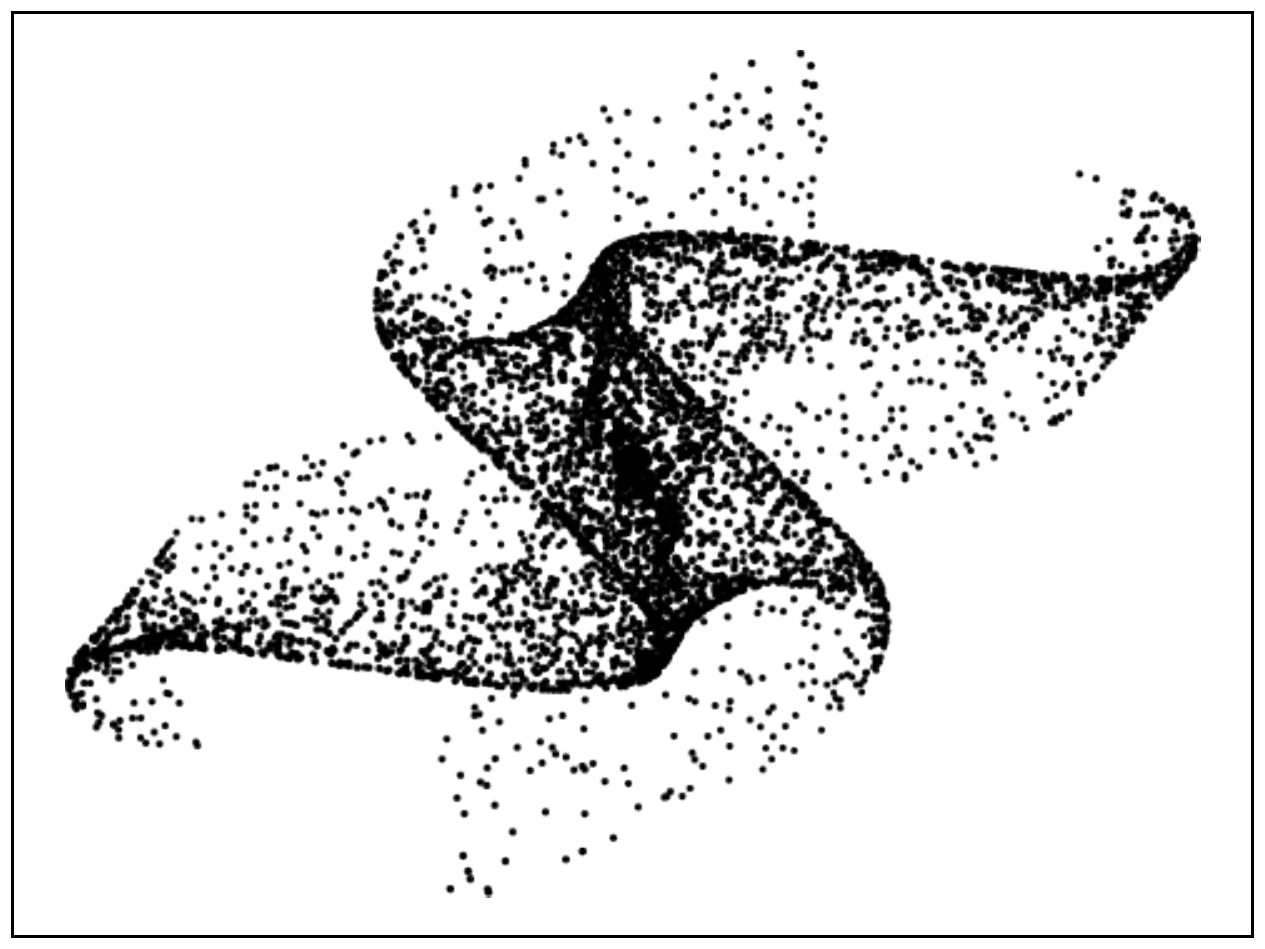}\hfill
\includegraphics[width=.24\linewidth]{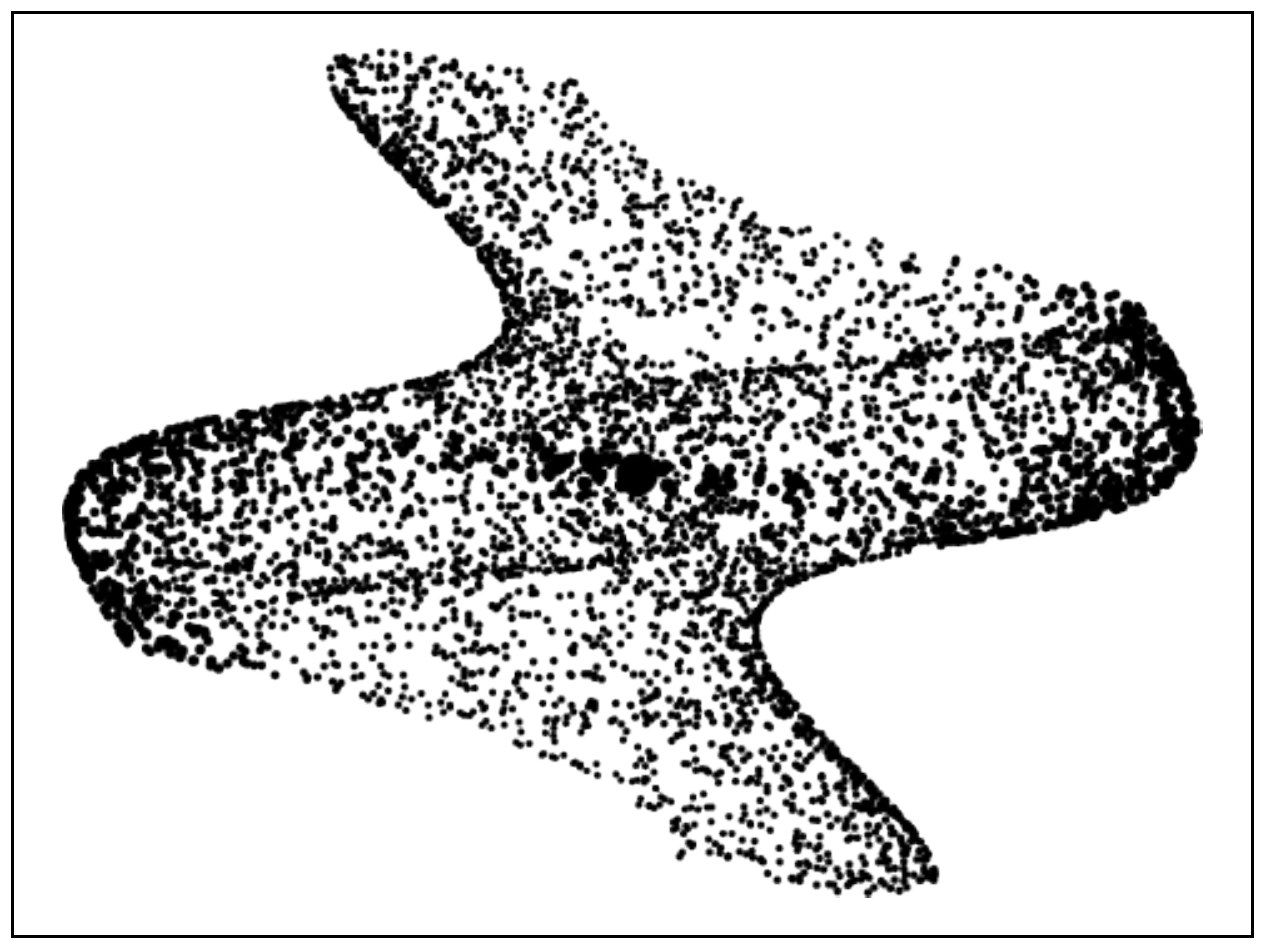}

\caption{\label{fig:embedding}5000 points sampled uniformly on $[-1,1]^2$ and mapped into $\R^6$ with the procedure described. The figures show the orthogonal projection of the mapped datapoints onto randomly sampled two-dimensional planes.}
\end{figure}

We want to investigate experimentally how the generalization performance behaves as a function of artificially added dimensions $p$. Similar experiments have been conducted, for example, in \cite{BickelLi_LocalPolynomialRegressuinOnUnknownManifolds} where the authors consider the synthetic three-dimensional dataset, where $x_1$ is sampled from a standard normal distribution and $x_2=x_1^3+\sin(x_1)-1$ and $x_3=\log(x_1^2+1)-x_1$. The response variable is $y=\cos(x_1)+x_2-x_3^3+\varepsilon$, where $\varepsilon$ is sampled from a normal distribution. They compare the performance of a local polynomial regressor as an estimator based on the whole feature vector $(x_1,x_2,x_3)$ against an estimator having only access to the only true feature $x_1$. In \cite{YangDunson_BayesianManifoldRegression} the authors consider datapoints lying on the two-dimensional \emph{swiss roll} manifold in $\R^3$, which they map into $\R^{100}$ via a random $100\times 3$-matrix and modeled the response variables as a function of the features plus noise. They only state, that their estimator "has a relatively fast convergence rate even though the dimension of the ambient space is large", but do not compare it to the performance of their estimator using the dataset in the original three-dimensional space, or a dataset in which the feature vectors contain only the two necessary parameters to describe the manifold. In \cite{NakadaImaizumi_AdaptiveApproximation} the authors conduct similar experiments using deep neural neetworks with ReLU activation function and the least-squares loss. They sample points from a uniform distribution on a $d'$-dimensional sphere in a $d$-dimensional space and modeled the response using a predefined function plus noise and examine the performance of a neural network for varying values valus of $d'$ and $d$ and different sample sizes. The hypothesis of low intrinsic dimensionality is especially prevalent for image datasets and convolutional neural network being able to exploit these structures. Although these highly specific datasets are not readily comparable to our setting, \cite{PopeZhuAbdelkaderGoldblumGoldstein_IntrinsicDimensionOfImages} consider a conceptually similar experimental setup where they keep the intrinsic dimension a dataset fixed and investigate the generalization performance for varying ambient dimensions.

%

For our purposes, we collected 32 regression datasets and 32 binary classification datasets from the UCI Repository \cite{UCIrepository} summarized in Tables \ref{tbl:reg_datasets} and \ref{tbl:class_datasets} in Section \ref{sec:dataset_summaries}. For the respective 16 smallest datasets we used a global kernel (i.e.~no partition), for the remaining datasets we used a partition such that each cell contains at most 4000 samples. For each dataset we performed training runs with the embedding described above for $p=0,\ldots,50$, where 20\% of each dataset was left out for testing and each run was repeated 10 times. For training and testing we used the command line version of liquidSVM \cite{SteinwartThomann_liquidSVM}, which implements a partitioning method. For hyperparameter selection we used 5-fold cross validation over a default $10\times 10$-grid chosen by liquidSVM based on some characteristics of the dataset, which has been empirically verified to yield competitive performance.  The results are summarized in Figures \ref{fig:results-glob-reg}, \ref{fig:results-glob-class}, \ref{fig:results-loc-reg}, and \ref{fig:results-loc-class} for regression with global kernels, classification with global kernels, regression with local kernels, and classification with local kernels respectively. We can divide the results in roughly three categories:
\begin{enumerate}
\item In accordance with our theoretical findings, the generalization performance is independent of the number of artificially added dimensions. That is, the test error for the datasets $D_p, p=1,\ldots,50,$ is similar as for the original dataset $D$. The datasets in this category constitute a clear majority.
\item After an initial increase of the test error, the test error quickly levels out at a moderately higher test error, which is still well below the naive error, see Section \ref{sec:dataset_summaries}. We still see this as a partial verification of our theoretical findings since at least after a certain point, the test error is independent of the further artificially added dimensions. Examples of datasets in this category are \texttt{bike\_sharing\_casual}, \texttt{bike\_sharing\_total}, \texttt{gas\_sensor\_drift\_class}, \texttt{gas\_sensor\_drift\_conc},\\ \texttt{sml2010\_dining}, \texttt{sml2010\_room}, \texttt{thyroid\_ann}, and \texttt{travel\_review\_ratings}.
\item On a few rare exceptions, the test error grows significantly, as for the datasets \texttt{chess}, \texttt{crowd\_sourced\_mapping}, and \texttt{electrical\_grid\_stability\_simulated}.
\end{enumerate}

\begin{figure}[h!]
\centering
\includegraphics[width=\textwidth]{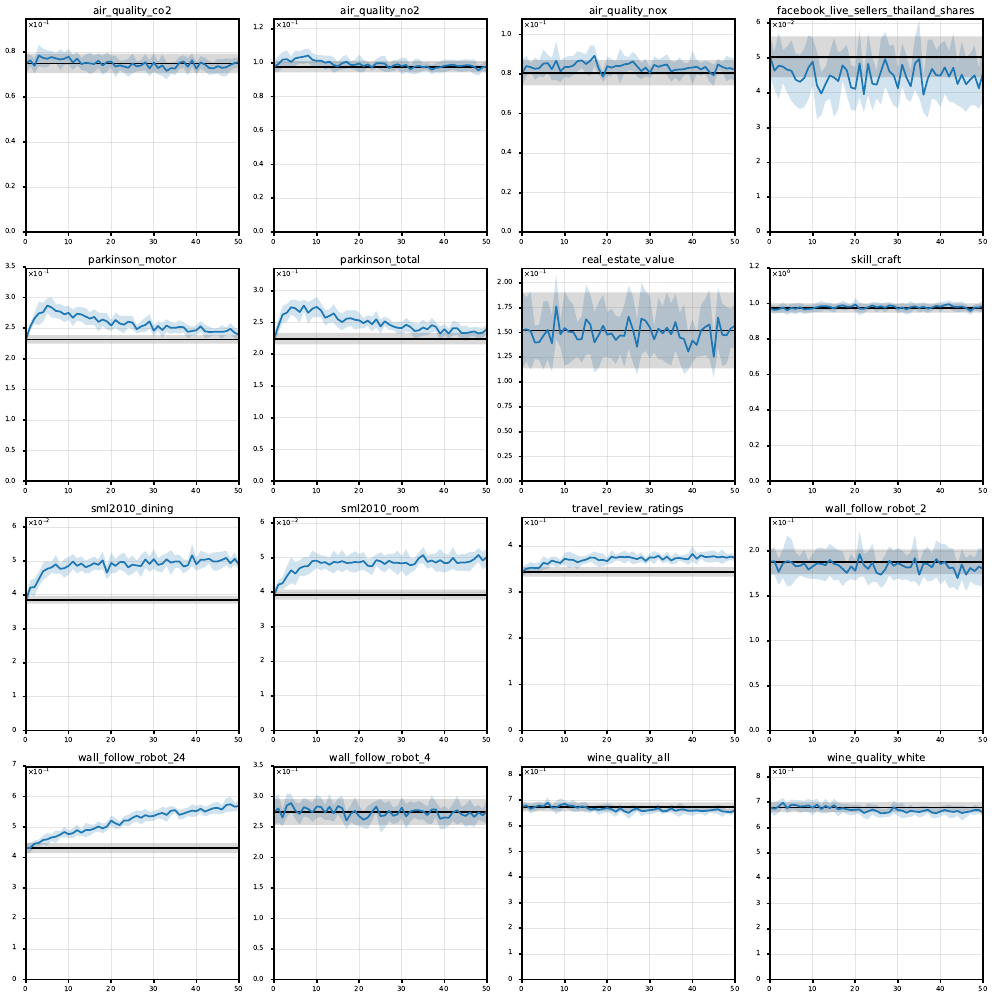}

\caption{\label{fig:results-glob-reg}Test root mean-squared errors ($y$-axis) for global kernels. The $x$-axis contains the number of artificially added dimensions. The shaded blue area corresponds to the standard deviation across the different runs. For comparison, the horizontal black line shows the test error for the original dataset. The shaded grey area corresponds to the standard deviation across the different runs for the original dataset.}
\end{figure}

\begin{figure}[h!]
\centering
\includegraphics[width=\textwidth]{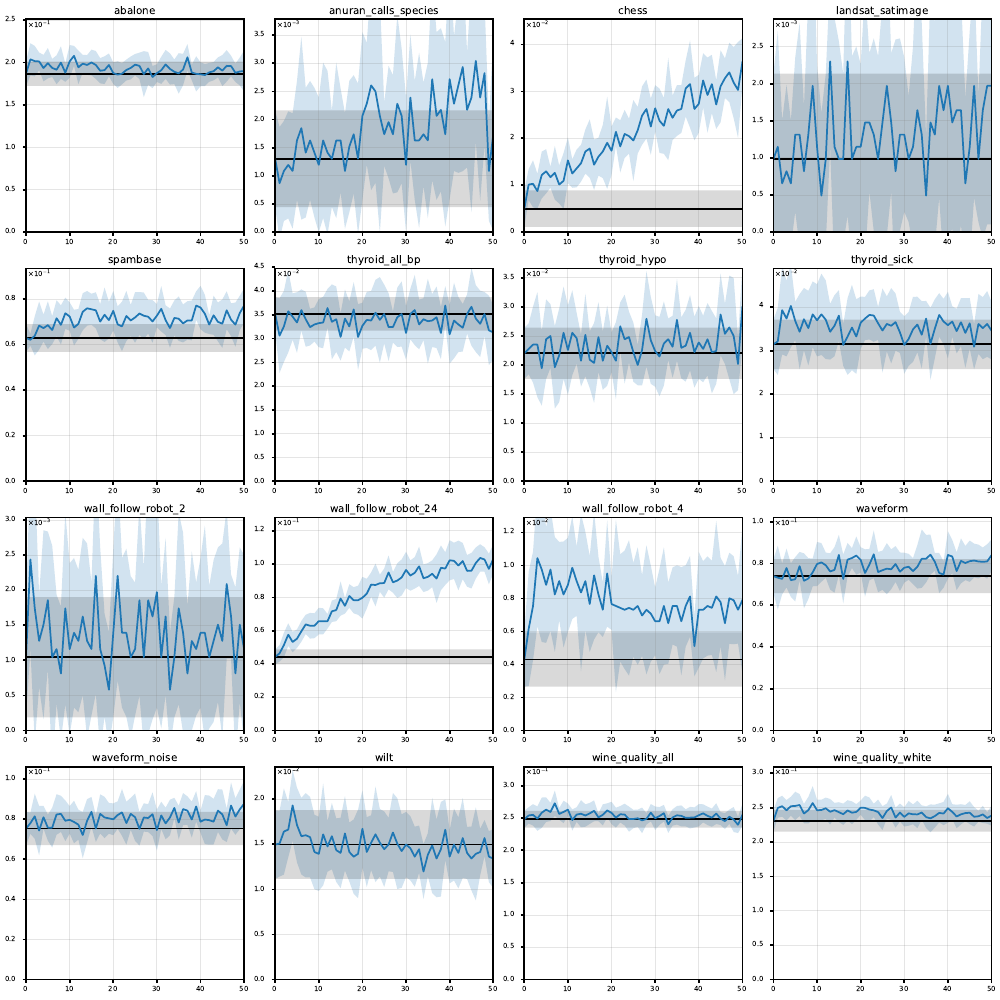}

\caption{\label{fig:results-glob-class}Test classification errors ($y$-axis) for global kernels. The $x$-axis contains the number of artificially added dimensions. The shaded blue area corresponds to the standard deviation across the different runs. For comparison, the horizontal black line shows the test error for the original dataset. The shaded grey area corresponds to the standard deviation across the different runs for the original dataset.}
\end{figure}

\begin{figure}[h!]
\centering
\includegraphics[width=\textwidth]{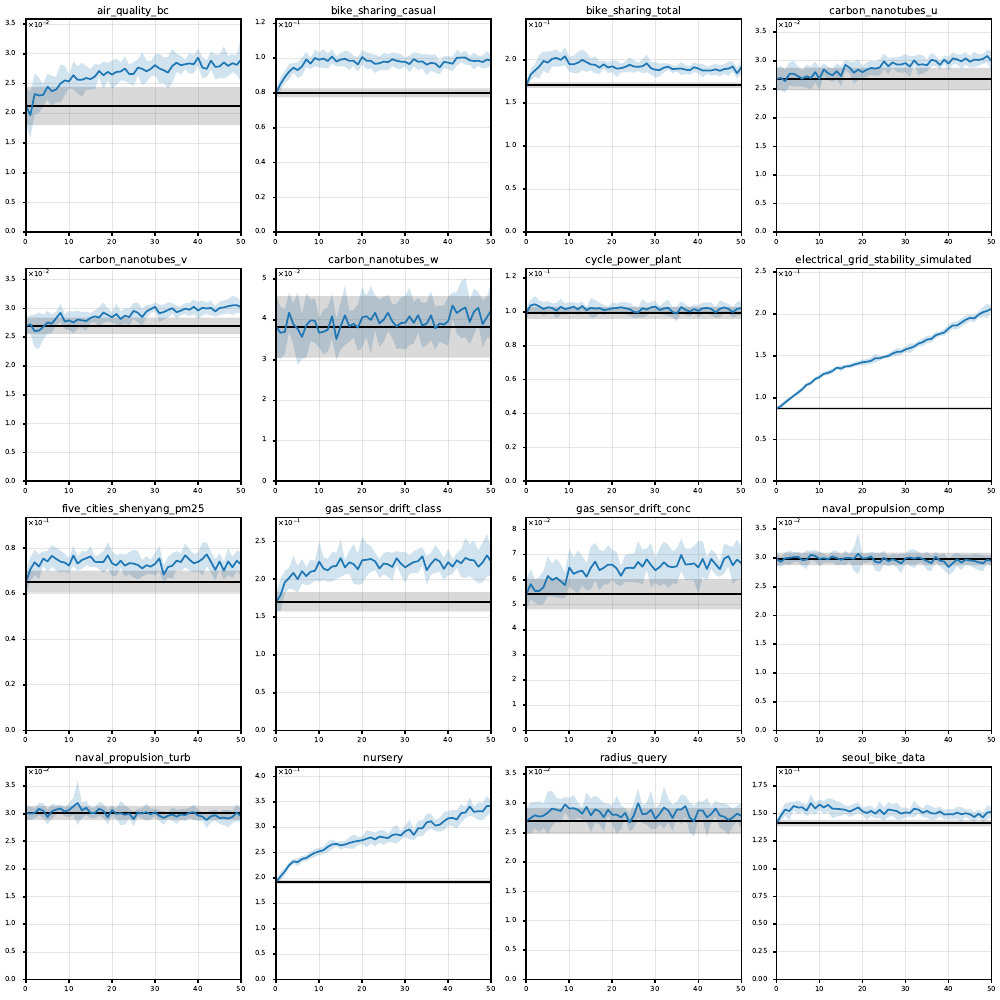}

\caption{\label{fig:results-loc-reg}Test root mean-squared errors ($y$-axis) for local kernels. The $x$-axis contains the number of artificially added dimensions. The shaded blue area corresponds to the standard deviation across the different runs. For comparison, the horizontal black line shows the test error for the original dataset. The shaded grey area corresponds to the standard deviation across the different runs for the original dataset.}
\end{figure}

\begin{figure}[h!]
\centering
\includegraphics[width=\textwidth]{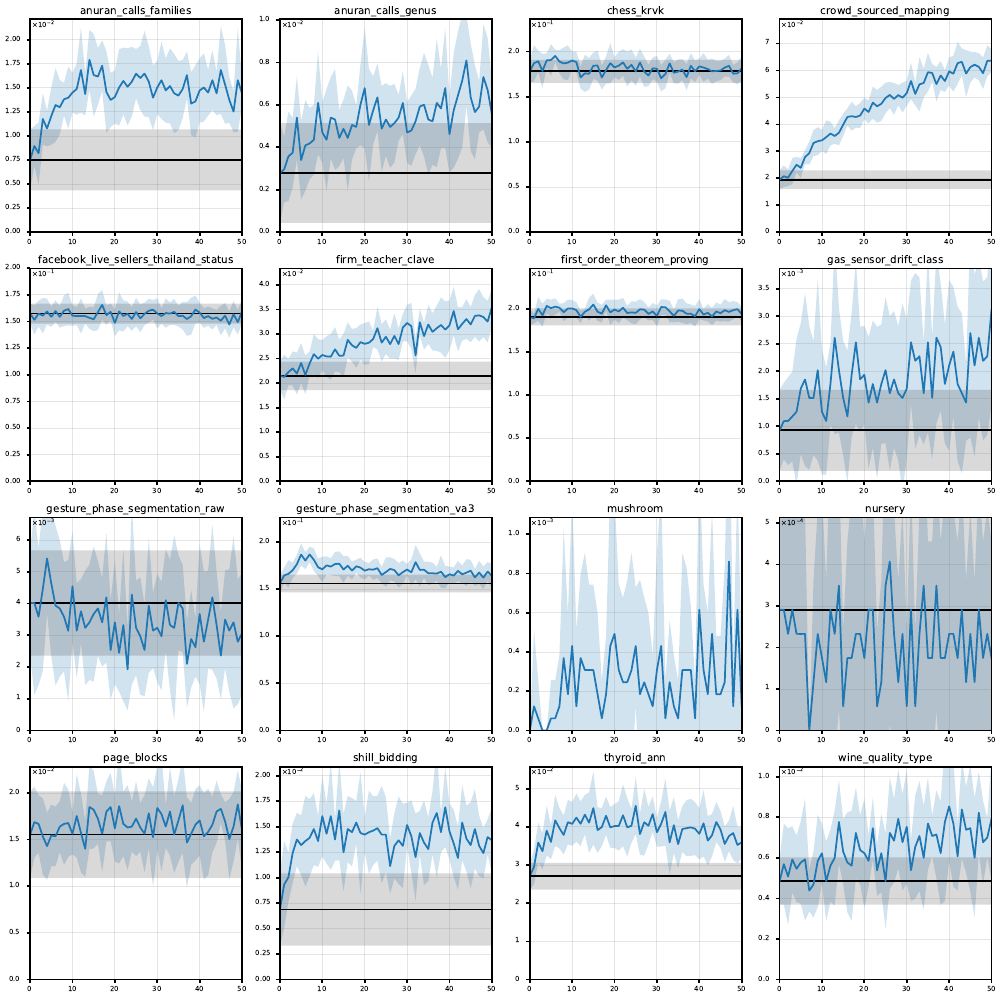}

\caption{\label{fig:results-loc-class}Test classification errors ($y$-axis) for local kernels. The $x$-axis contains the number of artificially added dimensions. The shaded blue area corresponds to the standard deviation across the different runs. For comparison, the horizontal black line shows the test error for the original dataset. The shaded grey area corresponds to the standard deviation across the different runs for the original dataset.}
\end{figure}

\section*{Acknowledgements}
The authors thank the International Max Planck Research School for Intelligent Systems (IMPRS-IS) for supporting Thomas Hamm. Thomas Hamm and Ingo Steinwart were supported by the German Research Foundation under DFG Grant STE 1074/4-1. 

\bibliography{bib}

\begin{appendices}
\section{RKHS Fundamentals}\label{sec:RKHS_fundamentals}
\begin{dfn}[Kernel]\label{dfn:kernel}
Let $X$ be a non-empty set. A function $k:X\times X\to\R$ is called a \emph{kernel} if there exists a (real) Hilbert space $H_0$ and a map $\Phi:X\to H_0$ such that $k(x,y)=\langle \Phi(x),\Phi(y)\rangle_{H_0}$ for all $x,y\in X$. The map $\Phi$ is called a \emph{feature map} and $H_0$ a \emph{feature} space of $k$.
\end{dfn}
By \cite[Theorem 4.16]{SteinwartChristmannSVMs}, a symmetric function $k:X\times X\rightarrow\R$ is a kernel if and only if it is positive definite, that is if for all $n\in\N$ and all choices $x_1,\ldots,x_n\in X$ and $\alpha_1,\ldots,\alpha_n\in\R$ we have
\begin{equation*}
\sum_{i=1}^n \sum_{j=1}^n \alpha_i \alpha_j k(x_j,x_i)\geq 0.
\end{equation*}
\begin{dfn}[Reproducing Kernel/Reproducing Property]\label{dfn:reproducing_kernel}
Let $X$ be a non-empty set and $H$ a Hilbert space consisting of function $f:X\to \R$. A function $k:X\times X\to\R$ is a \emph{reproducing kernel} if $k(x,\cdot)\in H$ for all $x\in X$ and
\begin{equation}\label{eqn:reproducing_property}
f(x)=\langle f, k(x,\cdot)\rangle \quad \text{for all } f\in H,x\in X.
\end{equation}
Property (\ref{eqn:reproducing_property}) is called the \emph{reproducing property}.
\end{dfn}
\begin{dfn}[RKHS]\label{dfn:RKHS}
Let $X$ be a non-empty set and $H$ a Hilbert space consisting of function $f:X\to \R$. Then $H$ is called a \emph{reproducing kernel Hilbert space} (RKHS) if the evaluation functional $H\to\R$ defined by $f\mapsto f(x)$ is continuous for every $x\in X$.
\end{dfn}
Definitions \ref{dfn:kernel}, \ref{dfn:reproducing_kernel}, and \ref{dfn:RKHS} are connected in the following way: Every reproducing kernel in the sense of Definition \ref{dfn:reproducing_kernel} is a kernel in the sense of Definition \ref{dfn:kernel} via the \emph{canonical feature map} $\Phi(x):=k(x,\cdot)$, see \cite[Lemma 4.19]{SteinwartChristmannSVMs}. Additionally, every RKHS has a unique reproducing kernel  \cite[Theorem 4.20]{SteinwartChristmannSVMs}. Conversely, every kernel $k$ has a unique RKHS $H$, for which it is the reproducing kernel, consisting of the functions $x\mapsto \langle \Phi(x),w\rangle_{H_0}, w\in H_0$, where $\Phi:X\to H_0$ is a feature map of $k$ and the norm in $H$ is given by
\begin{equation}\label{eqn:canonical_RKHS}
\norm{f}_H=\inf\left\{\norm{w}_{H_0}: w\in H_0 \text{ with } f=\langle \Phi(\cdot ),w\rangle\right\},
\end{equation}
see \cite[Theorem 4.21]{SteinwartChristmannSVMs}.

The following two results were already used in \cite{HammSteinwart_IntrinsicDimension}. As these results are crucial for the construction of localized kernels and their proofs are simple and instructive, we will repeat them at this point.
\begin{lem}\label{lem:transformed_kernel}
Let $k$ be a kernel on $X$ with RKHS $H$ and let $\psi:Y\rightarrow X$ be a map. Then $k_\psi(\cdot,\cdot):=k(\psi(\cdot),\psi(\cdot))$ is a kernel on $Y$ with RKHS $H_\psi=\{f\circ\psi:f\in H\}$ and the map $V:H\rightarrow H_\psi$ defined by $ f\mapsto f\circ\psi$ is a metric surjection. The norm in $H_\psi$ can be computed by
\begin{align*}
\norm{g}_{H_\psi}=\inf\{\norm{f}_H: f \text{ with } g=f\circ \psi \}.
\end{align*}
If $\psi$ is bijective, then $V$ is an isometric isomorphism.
\end{lem}
\begin{proof}
Let $\Phi:X\rightarrow H, x\mapsto k(x,\cdot)$ be the canonical feature map of $k$ and define $\Phi_\psi:Y\rightarrow H, y\mapsto \Phi(\psi(y))$. Then by construction we have $\langle \Phi_\psi(y),\Phi_\psi(y') \rangle_H=k_\psi(y,y')$ for all $y,y'\in Y$, that is, $\Phi_\psi$ is a feature map of $k_\psi$. The first two assertions now follow from (\ref{eqn:canonical_RKHS}). For the third assertion additionally apply this result on $\psi^{-1}$.
\end{proof}
\begin{cor}\label{cor:restricted_kernel}
Let $k$ be a kernel on $X\subset\R^d$, $H$ its RKHS and $Y\subset X$. Then $H|_Y:=\{f|_Y:f\in H\}$ is the RKHS of $k|_{Y\times Y}$ and the restriction $H\rightarrow H|_Y$ is a metric surjection.
\end{cor}
\begin{proof}
This follows from Lemma \ref{lem:transformed_kernel} with $\psi:Y\rightarrow X$ being the inclusion.
\end{proof}
\section{Entropy and Covering numbers}
Given a set $A\subset \R^d$ then by definition of $\varepsilon_m(A)$ for every $\varepsilon>\varepsilon_m(A)$ there exists an $\varepsilon$-net $N\subset A$ of $A$ with $|N|=m$. A useful property is, that for compact $A$ this also holds for $\varepsilon=\varepsilon_m(A)$, which is the content of the following lemma.
\begin{lem}\label{lem:existence_eps_net}
Let $A\subset\R^d$ be compact. Then for every $m\in\N$ there exists an $\varepsilon_m(A)$-net $N\subset A$ of $A$ with $|N|=m$.
\end{lem}
\begin{proof}
For $n\in\N$ let $x_{1,n},\ldots,x_{m,n}\in A$ be a $(\varepsilon_m(A)+1/n)$-net of $A$. By compactness of $A$, each sequence $(x_{j,n})_{n\in\N}$ has an accumulation point $x_j\in A, j=1,\ldots,m$. These accumulation points are an $\varepsilon_m(A)$-net, since for all $x\in A$ we have
\begin{align*}
\min_{j=1,\ldots,m} \norm{x-x_j} &\leq  \min_{j=1,\ldots,m} \norm{x-x_{j,n}} +\norm{x_{j,n}-x_j} \\
& \leq \min_{j=1,\ldots,m} \varepsilon_m(A)+\frac{1}{n} +\norm{x_{j,n}-x_j} \\
&= \varepsilon_m(A)+\frac{1}{n} + \min_{j=1,\ldots,m} \norm{x_{j,n}-x_j}.
\end{align*}
Taking the infimum over $n\in\N$ then yields the assertion.
\end{proof}
\begin{lem}\label{lem:entropy_subset}
For $A\subset B\subset \R^d$ we have $\varepsilon_m(A)\leq 2\varepsilon_m(B)$ for all $m\in\N$.
\end{lem}
\begin{proof}
Let $x_1,\ldots,x_m\in B$ be an $\varepsilon$-net of $B$. For each $j=1,\ldots,m$ pick an $y_j\in A$ with $\norm{x_j-y_j}\leq \varepsilon$, if such an $y_j$ exists and else let $y_j\in A$ be an arbitrary point. Then, by the triangle inequality $y_1,\ldots,y_m$ is a $2\varepsilon$-net of $A$.
\end{proof}
\begin{lem}\label{lem:entropy_vs_diam}
Let $A\subset\R^d$ be compact. Then we have $\varepsilon_m(A)\leq\diam\,A$ for all $m\in\N$.
\end{lem}
\begin{proof}
By monotonicity of $\varepsilon_m(A)$ it suffices to prove the statement for $m=1$. Let $x\in A$ with $A\subset B_\varepsilon(x)$ for $\varepsilon=\varepsilon_1(A)$, cf.~Lemma \ref{lem:existence_eps_net}. Then we have
\begin{equation*}
\varepsilon= \sup_{y\in A}\norm{x-y}\leq \sup_{y,z\in A} \norm{z-y}=\diam\, A.
\end{equation*}
\end{proof}
\begin{dfn}
Given normed spaces $E,F$ and a bounded, linear operator $T:E\to F$ the $i$-th dyadic entropy number of $T$ is defined as
\begin{align*}
e_i(T):=\inf \left\{ \varepsilon>0:\exists x_1,\ldots,x_{2^{i-1}}\in F \text{ such that } TB_E\subset\bigcup_{j=1}^{2^{i-1}}(x_j+\varepsilon B_F)\right\}.
\end{align*}
Further, the covering numbers of the operator $T$ are defined by $\cN(T,\varepsilon):=\cN(TB_E,\varepsilon)$ for $\varepsilon>0$.
\end{dfn}
\begin{lem}\label{lem:equivalence_entropy_covering_numbers}
Let $E,F$ be normed spaces and let $T:E\to F$ be a bounded, linear operator.
\begin{enumerate}
\item If there exist constants $a>0$ and $q>0$ such that $e_i(T)\leq a \,i^{-1/q}$ for all $i\in\N$, then we have
\begin{equation*}
\log\mathcal{N}(T,\varepsilon)\leq \log 4 \left( \frac{a}{\varepsilon}\right)^q \quad \text{for all } \varepsilon>0.
\end{equation*}
\item If there exist constants $a>0$ and $q>0$ such that $\log\mathcal{N}(T,\varepsilon)\leq (a/\varepsilon)^q$ for all $\varepsilon>0$, then we have
\begin{equation*}
e_i(T)\leq 3^\frac{1}{q} a \,i^{-\frac{1}{q}} \quad \text{for all }i\in\N.
\end{equation*}
\end{enumerate}
\end{lem}
The first assertion is the statement of \cite[Lemma 6.21]{SteinwartChristmannSVMs}, the second assertion is the content of \cite[Exercise 6.8]{SteinwartChristmannSVMs}.
\section{A General Oracle Inequality}
In this section we will proof a general oracle inequality for regularized empirical risk minimizers under Assumptions \ref{ass:assouad_dim} and \ref{ass:marginal} for bounded loss functions satisfying a so-called variance bound.
\begin{dfn}\label{dfn:variance_bound}
Let $L$ be a loss that can be clipped at $M>0$ and let $\mathcal{F}$ be some function class of measurable functions $f:X\rightarrow\R$. Assume there exists a Bayes decision function $f_{L,\P}^*:X\rightarrow[-M,M]$. We say, that a supremum bound is satisfied, if there exists a constant $B>0$, such that $L(y,t)\leq B$ for all $(y,t)\in Y\times[-M,M]$. We further say, that a variance bound is satisfied, if there exist $\vartheta\in[0,1]$ and $V\geq B^{2-\vartheta}$, such that
\begin{align*}
\E (L\circ \wideparen{f}-L\circ f_{L,\P}^*)^2 \leq V\cdot \left( \E \,L\circ \wideparen{f}-L\circ f_{L,\P}^*\right)^\vartheta
\end{align*}
for all $f\in\mathcal{F}$.
\end{dfn}

In the following, we first collect some preliminary results which are mainly used to bound the term
\begin{equation*}
\E_{D\sim\mu^n}e_i(\id:\LRKHS \rightarrow L_2(\D))\nonumber,
\end{equation*}
which in turn is our main tool for bounding the statistical error of our estimator $\wideparen{f}_{D,\blambda,\bgamma}$. To this end, first recall that we assume that $\mathcal{A}=(A_j)_{j=1,\ldots,m}$ is a Voronoi partition of our input space $X\subset\R^d$ with respect to the centers $C=\{c_1,\ldots,c_m\}$ constructed by the FFT algorithm \ref{alg:FFT} based on a sample drawn from $\mu^n$.

\begin{lem}\label{lem:empirical_distance}
Let Assumption \ref{ass:marginal} be satisfied for the constants $C_\mu$ and $\delta$. Then we have
\begin{align*}
\mu^n\left( x_1,\ldots,x_n:\sup_{x\in S} \min_{i=1,\ldots,n} \norm{x_i-x}>\tau \right)\leq m\exp\big(-C_\mu^{-1}(\tau-\varepsilon_m(S))^\delta n\big)
\end{align*}
for all $\varepsilon_m(S)<\tau\leq\varepsilon_m(S)+\diam\, S$.
\end{lem}
\begin{proof}
First note that $\min_{i=1,\ldots,n}\norm{x_i-x}>\tau$ if and only if $\norm{x_i-x}>\tau$ for all $i=1,\ldots,n$, which implies
\begin{equation}
\mu^n\left(x_1,\ldots,x_n: \min_{i=1,\ldots,n}\norm{x_i-x}>\tau \right)=\big(1-\mu(B_\tau(x))\big)^n \label{eqn:ball_proba}
\end{equation}
for all $x\in S$ and $\tau>0$. With the help of Lemma \ref{lem:existence_eps_net} let $N\subset S$ be an $\varepsilon_m(S)$-net of $S$ with $|N|=m$. Now, for every $x\in S$ there exists an $x'\in N$ such that
\begin{equation*}
\min_{i=1,\ldots,n} \norm{x_i-x}\leq\min_{i=1,\ldots,n} \norm{x_i-x'}+\norm{x'-x}\leq\min_{i=1,\ldots,n} \norm{x_i-x'}+\varepsilon_m(S),
\end{equation*}
which combined with (\ref{eqn:ball_proba}) implies
\begin{align*}
\mu^n\left( \sup_{x\in S} \min_{i=1,\ldots,n} \norm{x_i-x}>\tau \right)&\leq \mu^n\left( \max_{x\in N} \min_{i=1,\ldots,n}\norm{x_i-x}+\varepsilon_m(S)>\tau \right) \\
&\leq \sum_{x\in N} \big( 1-\mu(B_{\tau-\varepsilon_m(S)}(x)) \big)^n 
\end{align*}
for $\tau-\varepsilon_m(S)>0$. Using Assumption \ref{ass:marginal} we can further bound this by
\begin{align*}
\sum_{x\in N} \big( 1-\mu(B_{\tau-\varepsilon_m(S)}(x)) \big)^n &\leq m \big(1-C_\mu^{-1} (\tau-\varepsilon_m(S))^\delta\big)^n \\
&\leq m \exp\big(-C_\mu^{-1}(\tau-\varepsilon_m(S))^\delta n\big)
\end{align*}
for $\tau-\varepsilon_m(S)\leq\diam\, S$, which proves the assertion.
\end{proof}
\begin{cor}\label{cor:prob_opt_covering}
Let Assumption \ref{ass:marginal} be satisfied. Then for all $m\leq n$ we have $A_j\cap S\subset B_{6\varepsilon_m(S)}(c_j)$ for all $j=1,\ldots,m$ simultaneously with probability not less than
\begin{equation*}
1-m\exp\big(-C_\mu^{-1} n \varepsilon_m(S)^\delta \big).
\end{equation*}
\end{cor}
\begin{proof}
For $x\in S$ let $c(x)\in C$ be its respective Voronoi center and let $D=\{x_1,\ldots,x_n\}$. Recall, that since the FFT algorithm produces a 2-approximation of the metric $k$-center problem, we have $\norm{x_i-c(x_i)}\leq 2\varepsilon_m(D)$ for all $i=1,\ldots,n$. Consequently, we can estimate
\begin{align*}
\norm{x-c(x)}&= \min_{i=1,\ldots, n}\norm{x-c(x_i)}\leq\min_{i=1,\ldots,n} \norm{x-x_i}+\norm{x_i-c(x_i) } \\
&\leq \min_{i=1,\ldots,n} \norm{x-x_i}+2\varepsilon_m(D)\leq \min_{i=1,\ldots,n} \norm{x-x_i}+4\varepsilon_m(S),
\end{align*}
where in the last step we used Lemma \ref{lem:entropy_subset}. Applying Lemma \ref{lem:empirical_distance} with $\tau=2\varepsilon_m(S)$ subsequently gives us
\begin{equation*}
\sup_{x\in S}\norm{x-c(x)}\leq 6\varepsilon_m(S)
\end{equation*}
with probability not less than
\begin{equation*}
1-m\exp\big(-C_\mu^{-1} n \varepsilon_m(S)^\delta \big).
\end{equation*}
Note that the prerequisite $\varepsilon_m(S)<\tau\leq\varepsilon_m(S)+\diam\, S$ of Lemma \ref{lem:empirical_distance} is fulfilled for $\tau=2\varepsilon_m(S)$ because of Lemma \ref{lem:entropy_vs_diam}.
\end{proof}
\begin{lem}\label{lem:entropy_local_kernel}
Assume there exist constants $a_1,\ldots,a_m>0$ and $q>0$ such that
\begin{align*}
e_i(\id:H_{\gamma_j}(A_j)\rightarrow \ell_\infty(A_j\cap S)) \leq a_j\, i^{-\frac{1}{q}}.
\end{align*}
for all $i\in\N$ and $j=1,\ldots,m$. Then we have
\begin{align*}
e_i(\id:\LRKHS\rightarrow\ell_\infty(S))\leq (3\log 4)^{\frac{1}{q}}\left( \sum_{j=1}^m \left( \frac{a_j}{\sqrt{\lambda_j}} \right)^q \right)^\frac{1}{q} i^{-\frac{1}{q}}.
\end{align*}
for all $i\in\N$.
\end{lem}
\begin{proof}
By Lemma \ref{lem:equivalence_entropy_covering_numbers} we have
\begin{align*}
\log \cN_{\ell_\infty(A_j\cap S)}\left(B_{H_{\gamma_j}(A_j)},\varepsilon\right)\leq \log(4) \left(\dfrac{a_j}{\varepsilon}\right)^q
\end{align*}
for all $\varepsilon>0$ and hence
\begin{align*}
\log \cN_{\ell_\infty(A_j\cap S)}\left(\lambda^{-\frac{1}{2}}B_{H_{\gamma_j}(A_j)},\varepsilon\right)\leq \log(4) \left(\dfrac{a_j}{\varepsilon\sqrt{\lambda_j}}\right)^q,
\end{align*}
which yields
\begin{align*}
\log\cN_{\ell_\infty(S)}\left(B_{\LRKHS},\varepsilon\right)\leq& \sum_{j=1}^m\log \cN_{\ell_\infty(A_j\cap S)}\left( \lambda_j^{-\frac{1}{2}}B_{H_{\gamma_j}(A_j)},\varepsilon \right) \\
\leq& \sum_{j=1}^m\log(4)\left(\dfrac{a_j}{\varepsilon \sqrt{\lambda_j}}\right)^q
\end{align*}
where in the first estimate we used \cite[Lemma A.7]{HammSteinwart_IntrinsicDimension}. Finally, we again turn this into a bound on the dyadic entropy numbers using Lemma \ref{lem:equivalence_entropy_covering_numbers}, which completes the proof.
\end{proof}
\begin{lem}[{\cite[Theorem A.2]{HammSteinwart_IntrinsicDimension}}]\label{lem:entropy_numbers_gauss_rkhs}
There exists a universal constant $K$ only depending on $d$, such that
\begin{align*}
e_i(\id:H_\gamma(A)\rightarrow \ell_\infty(A))\leq K^{\frac{1}{2p}}p^{-\frac{d+1}{2p}}\cN(A,\gamma)^{\frac{1}{2p}}\,i^{-\frac{1}{2p}}
\end{align*}
holds for all $A\subset\R^d$, $i\in\N$, $p\in(0,1)$ and $\gamma>0$.
\end{lem}
\begin{cor}\label{cor:entroby_bound}
Let Assumptions \ref{ass:assouad_dim_ast} and \ref{ass:marginal} be satisfied. Then there exists a constant $K>0$ such that 
\begin{align*}
&\E_{D\sim\mu^n}e_i(\id:\LRKHS \rightarrow L_2(\D)) \\
\leq& (C_S^{\varrho+1}K)^\frac{1}{2p}p^{-\frac{d+1}{2p}} \lambda^{-\frac{1}{2}} \gamma^{-\frac{\varrho}{2p}} \left(1+m^{1+\frac{1}{2p}}\exp(-C_\mu^{-1} C_S^{-\varrho} n m^{-\delta/\varrho})\right)\, i^{-\frac{1}{2p}}
\end{align*}
for all $\gamma_1=\ldots=\gamma_m=\gamma\leq m^{-1/\varrho}$, $\lambda_1=\ldots=\lambda_m$, and $p\in(0,1)$.
\end{cor}
\begin{proof}
Let $Z$ be the set of samples $D\in X^n$ such that $A_j\cap S\subset B_{6\varepsilon_m(S)}(c_j)$ and decompose the expectation into
\begin{align}
&\E_{D\sim\mu^n}e_i(\id:\LRKHS \rightarrow L_2(\D))\nonumber \\
=&\, \E_{D\sim\mu^n}\mathbf{1}_Z e_i(\id:\LRKHS \rightarrow L_2(\D))+\E_{D\sim\mu^n}\mathbf{1}_{Z^C} e_i(\id:\LRKHS \rightarrow L_2(\D)).\label{eqn:expectation_split}
\end{align}
To estimate the first summand in (\ref{eqn:expectation_split}) we will use Lemma \ref{lem:entropy_local_kernel} with
\begin{equation}
a_j=K^\frac{1}{2p}p^{-\frac{d+1}{2p}} \cN(S\cap A_j,\gamma)^\frac{1}{2p},
\end{equation}
and $q=2p$, cf.~Lemma \ref{lem:entropy_numbers_gauss_rkhs}. To this end, note that $S \cap A_j\subset S\cap B_r(c_j)$, where $r:=\max \{m^{-1/\varrho},6\varepsilon_m(S)\}$. By Assumption \ref{ass:assouad_dim_ast} we have
\begin{align*}
a_j&\leq K^\frac{1}{2p} p^{-\frac{d+1}{2p}} \cN(S\cap B_r(c_j),\gamma)^\frac{1}{2p}\leq K^\frac{1}{2p} p^{-\frac{d+1}{2p}} \left( C_S\left( \frac{\gamma}{r}\right)^{-\varrho}\right)^\frac{1}{2p}  \\
&\leq  \left(6^\varrho C_S^{\varrho+1} K\right)^\frac{1}{2p} p^{-\frac{d+1}{2p}}\gamma^{-\frac{\varrho}{2p}} m^{-\frac{1}{2p}}
\end{align*}
for $\gamma\leq m^{-1/\varrho}\leq r$. This implies
\begin{align*}
\E_{D\sim\mu^n}\mathbf{1}_Z e_i(\id:\LRKHS \rightarrow L_2(\D)) &\leq \E_{D\sim\mu^n}\mathbf{1}_Z e_i(\id:\LRKHS \rightarrow \ell_\infty(S)) \\
&\leq \E_{D\sim\mu^n} (3\log 4)^\frac{1}{2p} \left( \sum_{j=1}^m \left(\frac{a_j}{\sqrt{\lambda}} \right)^{2p}\right)^{\frac{1}{2p}} i^{-\frac{1}{2p}} \\
&\leq  (3\cdot 6^\varrho\log (4) KC_S^{\varrho+1})^\frac{1}{2p} p^{-\frac{d+1}{2p}} \lambda^{-\frac{1}{2}} \gamma^{-\frac{\varrho}{2p}} i^{-\frac{1}{2p}}
\end{align*}
for $m\in\N$ and $\gamma\leq m^{-1/\varrho}$. For the second summand in (\ref{eqn:expectation_split}) we use Lemma \ref{lem:entropy_local_kernel} with $a_j=K^\frac{1}{2p}p^{-\frac{d+1}{2p}}\cN(S,\gamma)^\frac{1}{2p}$ and $q=2p$ and Corollary \ref{cor:prob_opt_covering} to bound $\mu^n(Z^C)$. Note that for $r_0:=\max\{1,\diam\,S\}$ we have by Lemma \ref{lem:entropy_numbers_gauss_rkhs} and Assumption \ref{ass:assouad_dim_ast}
\begin{align*}
a_j\leq  K^\frac{1}{2p}p^{-\frac{d+1}{2p}}\gamma^{-\frac{\varrho}{2p}} r_0^\frac{\varrho}{2p}
\end{align*}
for $\gamma\leq 1\leq r_0$, which implies
\begin{align*}
&\E_{D\sim\mu^n}\mathbf{1}_{Z^C} e_i(\id:\LRKHS \rightarrow L_2(\D)) \\
\leq & \mu^n(Z^C) (3\log 4)^\frac{1}{2p}\left( m K p^{-d-1}\gamma^{-\varrho}r_0^\varrho \lambda^{-p}\right)^\frac{1}{2p} i^{-\frac{1}{2p}} \\
\leq & m\exp(-C_\mu^{-1} n\varepsilon_m(S)^\delta) \left(3\log (4) m K p^{-d-1}\gamma^{-\varrho}r_0^\varrho \lambda^{-p}  \right)^\frac{1}{2p} i^{-\frac{1}{2p}} \\
\leq & (3\log(4)r_0^\varrho K)^\frac{1}{2p} m^{1+\frac{1}{2p}}\exp(-C_\mu^{-1} C_S^{-\delta} nm^{-\delta/\varrho})  p^{-\frac{d+1}{2p}}\lambda^{-\frac{1}{2}} \gamma^{-\frac{\varrho}{2p}} i^{-\frac{1}{2p}}
\end{align*}
for $\gamma\leq 1$, where in the second inequality we used Corollary \ref{cor:prob_opt_covering} to bound $\mu^n(Z^C)$ and in the last inequality we used the lower bound on $\varepsilon_m(S)$ from Assumption \ref{ass:assouad_dim_ast}.
\end{proof}
Finally, before we can present our general oracle inequality, we need to introduce a last regularity assumption on the considered loss function. To this end, we say that a loss function $L:Y\times\R\rightarrow [0,\infty)$ is locally Lipschitz continuous if for every $a>0$ the functions $L(y,\cdot)|_{[-a,a]}, y\in Y,$ are uniformly Lipschitz continuous, that is
\begin{align*}
|L|_{a,1}:=\sup_{\substack{s,t\in [-a,a], s\neq t \\y\in Y} } \frac{|L(y,t)-L(y,s)|}{|t-s|}<\infty.
\end{align*}
\begin{thm}\label{thm:general_oracle_inequality}
Assume $L$ is a locally Lipschitz continuous loss that can be clipped at $M>0$ and that the supremum and variance bounds are satisfied for constants $B>0, \vartheta\in[0,1]$, and $V\geq B^{2-\vartheta}$. Furthermore, assume \ref{ass:assouad_dim_ast} and \ref{ass:marginal} are satisfied and fix an $f_0\in \LRKHS$ and a $B_0\geq B$ with $\norm{L\circ f_0}_\infty\leq B_0$. Then there exists a constant $K$ such that for all $n\in\N, \gamma_1=\ldots=\gamma_m=:\gamma\in(0,m^{-1/\varrho}),\lambda_1=\ldots=\lambda_m=:\lambda>0, p\in (0,1/2]$ and $\tau>0$ we have
\begin{align*}
&\Rk_{L,\P}(\fftsvm)-\Rk_{L,\P}^* \\
\leq &\,9(\norm{f_0}_{\LRKHS}^2+\Rk_{L,\P}(f_0)-\Rk_{L,\P}^*) \\
&+ KC_{\P,m}\left( \frac{p^{-d-1}\gamma^{-\varrho}}{\lambda^p n} \right)^\frac{1}{2-p-\vartheta+\vartheta p} +3\left( \frac{72V\tau}{n} \right)^\frac{1}{2-\vartheta}+\frac{15B_0 \tau}{n}
\end{align*}
with probability $\P^n$ not less than $1-3\e^{-\tau}$, where
\begin{align*}
C_{\P,m}=&\max\left\{B, \left( |L|_{M,1}^pV^\frac{1-p}{2}\right)^\frac{2}{2-p-\vartheta+\vartheta p}, |L|_{M,1}^pB^{1-p},1 \right\} \\
&\cdot \max\left\{ C_S^{\varrho+1}\left(1+m^{2p+1}\exp\left(-2C_\mu^{-1} C_S^{-\delta}pnm^{-\delta/\varrho} \right)\right),B^{2p}\right\}^\frac{1}{2-p-\vartheta+\vartheta p}.
\end{align*}
\end{thm}
\begin{proof}
By \cite[Theorem 7.23]{SteinwartChristmannSVMs} together with the entropy estimate from Corollary \ref{cor:entroby_bound} we have
\begin{align*}
\Rk_{L,\P}(\fftsvm)-\Rk_{L,\P}^*\leq &\,9(\norm{f_0}_{\LRKHS}^2+\Rk_{L,\P}(f_0)-\Rk_{L,\P}^*) \\
&+ K(p) \left( \dfrac{a^{2p}}{ n} \right)^\frac{1}{2-p-\vartheta+\vartheta p}+3\left( \frac{72V\tau}{n} \right)^\frac{1}{2-\vartheta}+\frac{15B_0 \tau}{n}
\end{align*}
with probability $\P^n$ not less than $1-3\e^{-\tau}$, where
\begin{align*}
a :=\max\left\{ (C_S^{\varrho+1}K)^{\frac{1}{2p}}p^{-\frac{d+1}{2p}}\lambda^{-\frac{1}{2}}\gamma^{-\frac{\varrho}{2p}}\left(1+m^{1+\frac{1}{2p}}\exp(-C_\mu^{-1} C_S^{-\delta}nm^{-\delta/\varrho})\right),B \right\}
\end{align*}
and $K(p)$ satisfies
\begin{equation*}
K(p)\leq \tilde{K}\max\left\{ B, \left(|L|_{M,1}^p V^\frac{1-p}{2} \right)^\frac{2}{2-p-\vartheta+\vartheta p},|L|_{M,1}^pB^{1-p},1  \right\}
\end{equation*}
for a universal constant $\tilde{K}$, see \cite[Proof of Theorem A.10]{HammSteinwart_IntrinsicDimension}, for all $p\in(0,1/2]$. Simplifying the term $a^{2p}$ using that $(x+y)^{2p}\leq x^{2p}+y^{2p}$ for $x,y\geq 0$ and $p\in(0,1/2]$ then yields the result.
\end{proof}
\section{Proofs Related to Section \ref{sec:regression}}
\begin{proof}[Proof of Theorem \ref{thm:ls_oracle_inequality}]
The least-squares loss satisfies a supremum/variance bound for $B=4M^2, V=16M^2$ and $\vartheta=1$ as well as $|L|_{M,1}=4M$. Theorem \ref{thm:ls_oracle_inequality} gives us
\begin{align*}
\Rk_{L,\P}(\fftsvm)-\Rk_{L,\P}^*\leq &\,9(\norm{f_0}_{\LRKHS}^2+\Rk_{L,\P}(f_0)-\Rk_{L,\P}^*) \\
&+ KC_{\P,m}\frac{p^{-d-1}\gamma^{-\varrho}}{\lambda^p n} +\frac{(3456M^2+15B_0)\tau}{n}
\end{align*}
with probability not less than $1-3\e^{-\tau}$, where
\begin{equation*}
C_{\P,m}\leq \max\left\{ 16M^2,1\right\}\max\left\{ C_S^{\varrho+1}\left(1+m^{2p+1}\exp\left(-2C_\mu^{-1} C_S^{-\delta}pnm^{-\delta/\varrho}\right)\right),(4M^2)^{2p}\right\}.
\end{equation*}
Since $n\geq 2$ we have $p=\log 2/(2\log n)\leq 1/2$ and the second factor above can be bounded by
\begin{equation*}
\max\left\{ C_S^{\varrho+1}\left(1+m^2\exp\left(-\log 2C_\mu^{-1} C_S^{-\delta}nm^{-\delta/\varrho}/\log n\right)\right),4M^2\right\}
\end{equation*}
and we have $p^{-d-1}\lambda^{-p}\leq (2/\log 2)^{d+1}\log ^{d+1}n \lambda^{-1/\log n}$ for $\lambda\leq 1$. To complete the proof we need to pick a suitable function $f_0\in\LRKHS$. To this end, note that for $\gamma=\gamma_j,\lambda=\lambda_j, j=1,\ldots,m$ we have $H_\gamma(X)\subset \LRKHS$ and $\norm{f}_{\LRKHS}^2\leq m\lambda\norm{f}_{H_\gamma(X)}^2$ for all $f\in H_\gamma(X)$. By \cite[Lemma A.12, Proof of Proposition 3.2]{HammSteinwart_IntrinsicDimension} there exists an $f_0\in H_\gamma(X)$ with
\begin{align*}
\Rk_{L,\P}(f_0)-\Rk_{L,\P}^*\leq  \left( \frac{\Gamma\left(\frac{\alpha+d}{2} \right)}{\Gamma\left( \frac{d}{2}\right)} \right)^2 2^{-\alpha}d^{k}|f_0|_{C^{k,\beta}(\R^d)}^2\gamma^{2\alpha}
\end{align*}
and $\norm{f_0}_{H_\gamma(X)}^2\leq \pi^{-d/2}4^{k+1}\norm{f_0}_{L_2(\R^d)} \gamma^{-d}$ \cite[Lemma A.13]{HammSteinwart_IntrinsicDimension} which completes the proof.
\end{proof}
\begin{proof}[Proof of Corollary \ref{cor:ls_learning_rates}]
We apply Theorem \ref{thm:ls_oracle_inequality} with the specified values for $\lambda$ and $\gamma$. Examining the summands in the bound given in Theorem \ref{thm:ls_oracle_inequality}, ignoring constants for the moment, we see that for $m=\lceil n^\sigma\rceil$ and $\gamma=n^{-a}, \lambda=n^{-b}$ with $a=1/(2\alpha+\varrho), b\geq \sigma+(2\alpha+d)/(2\alpha+\varrho)$ we have
\begin{equation*}
m\lambda\gamma^{-d}\leq 2n^\sigma n^{-b}n^{ad}\leq 2n^{-\frac{2\alpha}{2\alpha+d\varrho}},
\end{equation*}
\begin{equation*}
\gamma^{2\alpha}=n^{-2\alpha a}=n^{-\frac{2\alpha}{2\alpha+d\varrho}}, 
\end{equation*}
\begin{equation*}
\lambda^{-1/\log n}\gamma^{-\varrho}n^{-1}\log^{d+1}n=n^{b/\log n}n^{a\varrho}n^{-1} \log^{d+1}n=\e^bn^{-\frac{2\alpha}{2\alpha+\varrho}}\log^{d+1}n.
\end{equation*}
That is, every summand is of the order of $n^{-2\alpha/(2\alpha+\varrho)}\log^{d+1}n$. To complete the proof we only need to check that the constant 
\begin{equation*}
c_{m,n}=\max\left\{ C_S^{\varrho+1}\left(1+m^{2}\exp\left( -\log 2C_\mu^{-1} C_S^{-\delta}nm^{-\delta/\varrho}/\log n\right)\right),4M^2\right\}\max\left\{16M^2,1 \right\}
\end{equation*}
in Theorem \ref{thm:ls_oracle_inequality} is uniformly bounded in $n$ for $m=\lceil n^\sigma\rceil$. To this end, note that
\begin{align*}
& m^2\exp\left(-\log 2C_\mu^{-1} C_S^{-\delta}nm^{-\delta/\varrho}/\log n\right) \\
\leq & 4n^{2\sigma}\exp\left(-2^{-\delta/\varrho}\log 2C_\mu^{-1} C_S^{-\delta}n^{1-\sigma\delta/\varrho}/\log n\right)
\end{align*}
uniformly bounded in $n>1$ since $1-\sigma\delta/\varrho>0$.
\end{proof}
\begin{proof}[Proof of Theorem \ref{thm:ls_adaptive_rates}]
By \cite[Theorem 7.2]{SteinwartChristmannSVMs}, an oracle inequality for empirical risk minimization, we have
\begin{align}
\begin{split}\label{eqn:ls_validation_erm}
\Rk_{L,\P}(\fftsvmcv)-\Rk_{L,\P}^*\leq&\, 6 \min_{(\blambda,\bgamma)\in \Lambda_n^m\times \Gamma_n^m} \left( \Rk_{L,\P}(\wideparen{f}_{D_1,\blambda,\bgamma,\mathrm{FFT}(m)})-\Rk_{L,\P}^* \right) \\
&+\frac{512M^2(\tau+\log(1+|\Lambda_n^m\times\Gamma_n^m|))}{n-l} \\
\leq &\,6\left( \Rk_{L,\P}(\wideparen{f}_{D_1,\blambda^*,\bgamma^*,\mathrm{FFT}(m)})-\Rk_{L,\P}^* \right)\\
&+\frac{2048M^2(\tau+\log(1+|\Lambda_n^m\times\Gamma_n^m|))}{n}
\end{split}
\end{align}
with probability $\P^{n-l}$ not less than $1-\e^{-\tau}$, where in the last step we picked values $\bgamma^*\in\Gamma_n^m$ and $\blambda^*\in\Lambda_n^m$ which we will specify in a moment. We again only consider $\lambda_1=\ldots=\lambda_m=:\lambda$ and $\gamma_1=\ldots=\gamma_m=:\gamma$. Since $\sigma<\varrho/\delta$, by Theorem \ref{thm:ls_oracle_inequality} there exists a constant $C>0$ independent of $\lambda,\gamma$ and $n$ (see also proof of Corollary \ref{cor:ls_learning_rates}) such that
\begin{align*}
\Rk_{L,\P}(\wideparen{f}_{D_1,\blambda,\bgamma,\mathrm{FFT}(m)})-\Rk_{L,\P}^* &\leq C\left( m\lambda\gamma^{-d}+\gamma^{2\alpha} +\lambda^{-1/\log n}\gamma^{-\varrho}l^{-1}\log^{d+1}n+\frac{\tau}{l}\right) \\
&\leq C\left( m\lambda\gamma^{-d}+\gamma^{2\alpha} +2\e^{\sigma+d}\gamma^{-\varrho}n^{-1}\log^{d+1}n+\frac{2\tau}{n}\right)
\end{align*}
with probability $\P^l$ not less than $1-3\e^{-\tau}$ for $\lambda\in\Lambda_n$ and $\gamma\in\Gamma_n\cap(0,m^{-1/\varrho})$. Since $A_n$ is an $1/\log n$-net of $(0,1]$ we have   $\Gamma_n\cap(0,m^{-1/\varrho})\neq\emptyset$ for $1-\sigma/\varrho>2/\log n$ and since $\sigma<\varrho/(2\alpha+\varrho)$ we can choose $a_*\in A_n$ such that $\gamma=n^{-a_*}\in (0,m^{-1/\varrho})$ and $1/(2\alpha+\varrho)\leq a_*\leq 1/(2\alpha+\varrho)+2/\log n$. That is, by choosing $\gamma=n^{-a_*}$ and $\lambda=n^{-\sigma-d}$ as $\bgamma^*,\blambda^*$ we have
\begin{align*}
\Rk_{L,\P}(\wideparen{f}_{D_1,\blambda^*,\bgamma^*,\mathrm{FFT}(m)})\leq C\left( n^{-\frac{2\alpha}{2\alpha+\varrho}}+\e^{c+d+4\alpha}n^{-\frac{2\alpha}{2\alpha+\varrho}}\log^{d+1} n+\frac{2\tau}{n}\right)
\end{align*}
with probability not less than $1-3\e^{-\tau}$. Combining this with (\ref{eqn:ls_validation_erm}) we get using $|\Lambda_n^m\times\Gamma_n^m|\lesssim\log^{2m}n$ that
\begin{align*}
\Rk_{L,\P}(\fftsvmcv)-\Rk_{L,\P}^* &\leq c_1\left( n^{-\frac{2\alpha}{2\alpha+\varrho}}\log^{d+1}n+\frac{\tau}{n}\right)+c_2\left( \frac{\tau}{n}+\frac{m\log \log n}{n}\right)
\end{align*}
with probability $\P^n$ not less than $(1-\e^{-\tau})(1-3\e^{-\tau})\geq(1-4\e^{-\tau})$. Noting that $m/n=2n^{\sigma-1}\leq 2n^{-2\alpha/(2\alpha+\varrho)}$ and some elementary transformations yield the assertion.
\end{proof}
\section{Proofs Related to Section \ref{sec:classification}}
\begin{proof}[Proof of Theorem \ref{thm:class_oracle_inequality}]
The supremum bound is obviously satisfied for $B=2$ and by \cite[Theorem 8.24]{SteinwartChristmannSVMs} the variance bound is satisfied for $V=6C_*^{q/(q+1)}$ and $\vartheta=q/(q+1)$. Furthermore, it is not hard to see that $|\Lhinge|_{1,1}=1$. Given this value for $\vartheta$, the exponent in Theorem \ref{thm:general_oracle_inequality} then reads
\begin{align*}
\frac{1}{2-p-\vartheta+\vartheta p}&=\frac{1}{2-p-\frac{q}{q+1}(1-p)}=\frac{q+1}{(2-p)(q+1)-q(1-p)} \\
&=\frac{q+1}{2q+2-pq-p-q+pq} = \frac{q+1}{q+2-p}.
\end{align*}
An application of Theorem \ref{thm:general_oracle_inequality} then gives us for $\lambda>0$, $\gamma\in(0,m^{-1/\varrho})$, and $p\in(0,1/2]$
\begin{align*}
\Rk_{L,\P}(\fftsvm)-\Rk_{L,\P}^*\leq&\, 9\left(  \norm{f_0}_{\LRKHS}^2+\Rk_{L,\P}(f_0)-\Rk_{L,\P}^*\right)\\
&+C_{\P,m} K\left( \frac{p^{-d-1}\lambda^{-p}\gamma^{-\varrho}}{n} \right)^\frac{q+1}{q+2-p}+3\left(\frac{432C_*^{\frac{q}{q+1}}\tau}{n}\right)^\frac{q+1}{q+2} \\
&+\frac{15B_0\tau}{n}
\end{align*}
with probability not less than $1-3\e^{-\tau}$. With the specified values for $B,V,\vartheta$ and $|\Lhinge|_{1,1}$ we see that the first factor of the constant $C_{\P,m}$ can be bounded by
\begin{align*}
&\max\left\{B, \left( |L|_{M,1}^pV^\frac{1-p}{2}\right)^\frac{2}{2-p-\vartheta+\vartheta p}, |L|_{M,1}^pB^{1-p},1 \right\} \\
=& \max\left\{ 2, \left( 6C_*^{\frac{q}{q+1}} \right)^\frac{(1-p)(q+1)}{q+2-p},2^{1-p} \right\} \leq \max\left\{ 2,6C_*^\frac{q}{q+1} \right\}.
\end{align*}
Noting that $(q+1)/(q+2-p)\leq 1$ we see that the second factor of $C_{\P,m}$ in Theorem \ref{thm:general_oracle_inequality} is bounded by
\begin{equation*}
\max\left\{C_S^{\varrho+1}\left(1+m^{2p+1}\exp\left( -2C_\mu^{-1} C_S^{-\varrho}pn/m\right) \right),2\right\}.
\end{equation*}
Choosing $p=\log 2/(2\log n)\leq 1/2$ gives us
\begin{equation*}
C_{\P,m}K\left( \frac{p^{-d-1}\lambda^{-p}\gamma^{-\varrho}}{n} \right)^{\frac{q+1}{q+2-p}}\leq c_{m,n} K\lambda^{-1/\log n}\left( \frac{\gamma^{-\varrho}}{n}\right)^\frac{q+1}{q+2}
\end{equation*}
for $\gamma^{-\varrho}/n\leq 1$ with $c_{m,n}$ defined as in the theorem. Finally, to bound the approximation error we need to pick a suitable function $f_0\in\LRKHS$. To this end, note that for $\gamma=\gamma_j,\lambda=\lambda_j, j=1,\ldots,m$ we have $H_\gamma(X)\subset \LRKHS$ with $\norm{f}_{\LRKHS}^2\leq m\lambda\norm{f}_{H_\gamma(X)}^2$ for all $f\in H_\gamma(X)$. By Equation (8.15) in \cite[Proof of Theorem 8.18]{SteinwartChristmannSVMs} there exists a function $f_0\in H_\gamma(X)$ with $\norm{f_0}_\infty\leq 1$ and
\begin{align*}
\lambda\norm{f_0}_{H_\gamma(X)}^2+\Rk_{L,\P}(f_0)-\Rk_{L,\P}^*\leq \frac{3^d}{\Gamma\left(\frac{d}{2}+1\right)} \lambda\gamma^{-d}+\frac{2^{1-\beta/2}\Gamma\left(\frac{\beta+d}{2}\right)}{\Gamma\left(\frac{d}{2}\right)}C_{**}\gamma^\beta.
\end{align*}
since $\norm{f_0}_\infty\leq 1$ we have $B_0=2$ which completes the proof.
\end{proof}
\begin{proof}[Proof of Corollary \ref{cor:class_learning_rates}]
This follows from Theorem \ref{thm:class_oracle_inequality} by plugging in the values for $\lambda$ and $\gamma$ as specified, where we only need to check that the specified $\gamma$ is in the admissible range required by Theorem \ref{thm:class_oracle_inequality} and that the constant $c_{m,n}$ is uniformly bounded, similar to the proof of Corollary \ref{cor:ls_learning_rates}. For the former, recall that $\gamma=n^{-a}$ with
\begin{equation*}
a=\frac{q+1}{\beta(q+2)+\varrho(q+1)}.
\end{equation*}
By assumption we have $\sigma/\varrho \leq a$ and obviously also $a\leq 1/\varrho$. This implies $\gamma\in[n^{-1/\varrho},m^{-1/\varrho}]$, as required by Theorem \ref{thm:class_oracle_inequality}. Concerning the constant $c_{m,n}$ note that
\begin{align*}
m^2\exp\left( -\log (2) C_\mu C_S^{-\varrho}nm^{-\delta/\varrho}/\log n\right)\leq 4n^{2\sigma}\exp\left( -2^{-\delta/\varrho}\log (2) C_\mu C_S^{-\delta}n^{1-\sigma\delta/\varrho}/\log n\right)
\end{align*}
which is uniformly bounded in $n\in\N$ since $1-\sigma\delta/\varrho>0$.
\end{proof}
\begin{proof}[Proof of Theorem \ref{thm:class_adaptive_rates}]
By \cite[Theorem 7.2]{SteinwartChristmannSVMs}, an oracle inequality for empirical risk minimization, we have
\begin{align}
\begin{split}\label{eqn:class_validation_erm}
&\Rk_{L,\P}(\fftsvmcv)-\Rk_{L,\P}^* \\
\leq&\, 6 \min_{(\blambda,\bgamma)\in \Lambda_n\times \Gamma_n} \left( \Rk_{L,\P}(\wideparen{f}_{D_1,\blambda,\bgamma,\mathrm{FFT}(m)})-\Rk_{L,\P}^* \right) \\
&+4\left(\frac{48C_*^\frac{q}{q+1}(\tau+\log(1+|\Lambda_n^m\times\Gamma_n^m|))}{n-l} \right)^\frac{q+1}{q+2} \\
\leq &\,6\left( \Rk_{L,\P}(\wideparen{f}_{D_1,\blambda^*,\bgamma^*,\mathrm{FFT}(m)})-\Rk_{L,\P}^* \right)\\
&+4\left( \frac{192C_*^\frac{q}{q+1}(\tau+\log(1+|\Lambda_n^m\times\Gamma_n^m|))}{n}\right)^\frac{q+1}{q+2}
\end{split}
\end{align}
with probability $\P^{n-l}$ not less than $1-\e^{-\tau}$, where in the last step we picked values $\bgamma^*\in\Gamma_n^m$ and $\blambda^*\in\Lambda_n^m$ which we will specify in a moment. We again set $\lambda_1=\ldots=\lambda_m=:\lambda$ and $\gamma_1=\ldots=\gamma_m=:\gamma$. By Theorem \ref{thm:class_oracle_inequality} there exists a constant $C>0$ independent of $\lambda,\gamma$ and $n$ (see also proof of Corollary \ref{cor:class_learning_rates}) such that
\begin{align*}
&\Rk_{L,\P}(\wideparen{f}_{D_1,\blambda,\bgamma,\mathrm{FFT}(m)})-\Rk_{L,\P}^* \\
\leq& C\left( m\lambda\gamma^{-d}+\gamma^\beta+\lambda^{-1/\log n} \left( \frac{\gamma^{-\varrho}}{l}\right)^\frac{q+1}{q+2}\log^{d+1}n +\left(\frac{\tau}{l} \right)^\frac{q+1}{q+2}+\frac{\tau}{l}\right)
\end{align*}
with probability not less than $1-3\e^{-\tau}$ for all $\lambda\in\lambda_n$ and $\gamma\in\Gamma_n\cap [l^{-1/\varrho},m^{-1/\varrho}]$. Note that
\begin{equation*}
\left[l^{-1/\varrho},m^{-1/\varrho}\right]\supset \left[\left( \frac{2}{n}\right)^{1/\varrho},n^{-\sigma/\varrho} \right]=\left[n^{-(1-\log 2/\log n)/\varrho},n^{-\sigma/\varrho} \right].
\end{equation*}
Since $A_n$ is an $1/\log n$-net of $(0,1]$ we have for
\begin{equation*}
\left( 1-\frac{\log 2}{\log n}\right)\frac{1}{\varrho}-\frac{\sigma}{\varrho}>\frac{2}{\log n},
\end{equation*}
which is equivalent to $n>\exp ((2\varrho+\log 2)/(1-\sigma))$, that $\Gamma_n\cap [l^{-1/\varrho},m^{-1/\varrho}]\neq\emptyset$. That is, we can choose $a_*\in A_n$ such that $\gamma=n^{-a_*}$ is in the admissible range and satisfies
\begin{equation*}
\frac{q+1}{\beta(q+2)+\varrho(q+1)}-\frac{2}{\log n}\leq a_*\leq \frac{q+1}{\beta(q+2)+\varrho(q+1)}+\frac{2}{\log n}.
\end{equation*}
Choosing $\gamma=n^{-a_*}$ and $\lambda=n^{-\sigma-d}$, we can finish the proof exactly as the proof of Theorem \ref{thm:ls_adaptive_rates} by combining the inequalities above.
\end{proof}
\section{Dataset Summaries}\label{sec:dataset_summaries}
Below, summaries of the datasets we used for our experiments are given. For each dataset, the number of samples, the dimension of the input space, the naive error, and the base error is given. The naive error is the best error one can achieve using a constant decision function. In regression, this corresponds to the standard deviation of the labels $y_1,\ldots,y_n$ and in classification this corresponds to the fraction of labels from the smaller class. The base error is the test error on the (unmodified) dataset averaged over 10 repetitions.
\begin{table}[h]
\small\centering
\begin{tabular}{c c c c c}
Name& Samples& Dimension& Naive Error&Base error\\ \hline
air\_quality\_bc&8991&10&0.2343&0.0212\\
air\_quality\_co2&7674&10&0.2463&0.0751\\
air\_quality\_no2&7715&10&0.2862&0.0976\\
air\_quality\_nox&7718&10&0.2884&0.0806\\
bike\_sharing\_casual&17379&12&0.2687&0.0801\\
bike\_sharing\_total&17379&12&0.3717&0.1707\\
carbon\_nanotubes\_u&10721&5&0.6304&0.0268\\
carbon\_nanotubes\_v&10721&5&0.6311&0.0270\\
carbon\_nanotubes\_w&10721&5&0.5782&0.0382\\
cycle\_power\_plant&9568&4&0.4521&0.0992\\
electrical\_grid\_stability\_simulated&10000&12&0.3883&0.0871\\
facebook\_live\_sellers\_thailand\_shares&7050&9&0.0769&0.0504\\
five\_cities\_shenyang\_pm25&19038&14&0.1306&0.0653\\
gas\_sensor\_drift\_class&13910&128&1.7285&0.1702\\
gas\_sensor\_drift\_conc&13910&128&0.3432&0.0542\\
naval\_propulsion\_comp&11934&14&0.5888&0.0299\\
naval\_propulsion\_turb&11934&14&0.6000&0.0302\\
nursery&12960&8&1.2356&0.1923\\
parkinson\_motor&5875&19&0.4716&0.2316\\
parkinson\_total&5875&19&0.4459&0.2246\\
radius\_query&10000&3&0.3755&0.0270\\
real\_estate\_value&414&6&0.2473&0.1521\\
seoul\_bike\_data&8760&14&0.3627&0.1414\\
skill\_craft&3338&18&1.4480&0.9727\\
sml2010\_dining&4137&17&0.3769&0.0386\\
sml2010\_room&4137&17&0.3790&0.0393\\
travel\_review\_ratings&5456&23&0.6278&0.3442\\
wall\_follow\_robot\_2&5456&2&1.0047&0.1882\\
wall\_follow\_robot\_24&5456&24&1.0047&0.4310\\
wall\_follow\_robot\_4&5456&4&1.0047&0.2747\\
wine\_quality\_all&6497&12&0.8732&0.6739\\
wine\_quality\_white&4898&11&0.8855&0.6802\\ \hline
\end{tabular}
\caption{\label{tbl:reg_datasets}Regression datasets.}
\end{table}

\begin{table}
\small\centering
\begin{tabular}{c c c c c}
Name& Samples& Dimension& Naive Error&Base error\\ \hline
abalone&2870&8&0.4676&0.1868\\
anuran\_calls\_families&6585&22&0.3288&0.0075\\
anuran\_calls\_genus&5743&22&0.2774&0.0028\\
anuran\_calls\_species&4599&22&0.2437&0.0013\\
chess&3196&36&0.4778&0.0050\\
chess\_krvk&8747&22&0.4795&0.1782\\
crowd\_sourced\_mapping&9003&28&0.1659&0.0195\\
facebook\_live\_sellers\_thailand\_status&6622&9&0.3525&0.1574\\
firm\_teacher\_clave&8606&16&0.4997&0.0215\\
first\_order\_theorem\_proving&6118&51&0.4175&0.1904\\
gas\_sensor\_drift\_class&5935&128&0.4930&0.0009\\
gesture\_phase\_segmentation\_raw&5719&19&0.4842&0.0040\\
gesture\_phase\_segmentation\_va3&5691&32&0.4816&0.1560\\
landsat\_satimage&3041&36&0.4959&0.0010\\
mushroom&8124&111&0.4820&0.0000\\
nursery&8588&8&0.4967&0.0003\\
page\_blocks&5242&10&0.0628&0.0155\\
shill\_bidding&6321&9&0.1068&0.0069\\
spambase&4601&57&0.3940&0.0629\\
thyroid\_all\_bp&3621&31&0.0434&0.0352\\
thyroid\_ann&7034&21&0.0523&0.0271\\
thyroid\_hypo&2700&25&0.0504&0.0220\\
thyroid\_sick&3621&31&0.0621&0.0314\\
wall\_follow\_robot\_2&4302&2&0.4874&0.0010\\
wall\_follow\_robot\_24&4302&24&0.4874&0.0443\\
wall\_follow\_robot\_4&4302&4&0.4874&0.0043\\
waveform&3353&21&0.4942&0.0739\\
waveform\_noise&3347&40&0.4945&0.0754\\
wilt&4839&5&0.0539&0.0150\\
wine\_quality\_all&4974&12&0.4298&0.2481\\
wine\_quality\_type&6497&11&0.2461&0.0048\\
wine\_quality\_white&3655&11&0.3986&0.2311\\ \hline
\end{tabular}
\caption{\label{tbl:class_datasets}Classification datasets.}
\end{table}
\end{appendices}
\end{document}